\documentclass{article}

\usepackage[utf8]{inputenc}
\usepackage[margin=1in]{geometry}

\usepackage{amsmath}               % great math stuff
\usepackage{amsfonts,amssymb}              % for blackboard bold, etc
\usepackage{amsthm}                % better theorem environments
\usepackage{color}
\usepackage{enumerate}
\usepackage{hyperref}
\usepackage{bbm}
\usepackage{graphicx}

\newtheorem{theorem}{Theorem}[section]
\newtheorem{lemma}[theorem]{Lemma}
\newtheorem{proposition}[theorem]{Proposition}
\newtheorem{corollary}[theorem]{Corollary}
\newtheorem{definition}{Definition}
\theoremstyle{remark}
\newtheorem{example}[theorem]{Example}
\newtheorem{remark}{Remark}

\newcommand{\vf}{\varphi}

\newcommand{\R}{\mathbb R}

\newcommand{\toP}{\stackrel{P}{\longrightarrow}}
\newcommand{\Z}{\mathbb Z}
\newcommand{\E}{\mathbf {E}}
\newcommand{\Corr}{\mathrm{Corr}}
\newcommand{\pr}{\mathbb{P}}
\newcommand{\Cov}{\mathbf{Cov}}
\newcommand{\Var}{\mathbf{Var}}

\renewcommand{\Re}{\operatorname{Re}}

\newcommand{\ind}{\mathbbm{1}}
\newcommand{\eqd}{\stackrel{d}{=}}
\newcommand{\N}{\mathbb N}

\title{Prediction of random variables by excursion metric projections}
\author{Vitalii Makogin\thanks{Institute of Stochastics,
Ulm University, Germany; vitalii.makogin@uni-ulm.de}, Evgeny Spodarev\thanks{Institute of Stochastics,
Ulm University, Germany; evgeny.spodarev@uni-ulm.de}}
\date{\today}

\begin{document}

\maketitle

\begin{abstract}
We use the concept of excursions for the prediction of random variables without any moment existence assumptions. To do so, an excursion metric on the space of random variables  is defined which appears to be a kind of a weighted $L^1$-distance. Using equivalent forms of this metric and the specific choice of excursion levels, we formulate the prediction problem as a minimization of a certain target functional which involves the excursion metric.  Existence of the solution and weak consistency of the predictor are discussed. An application to the extrapolation of stationary heavy-tailed random functions illustrates the use of the aforementioned theory. Numerical experiments with the prediction of Gaussian, $\alpha$-stable and further heavy--tailed time series round up the paper.
\end{abstract}

 \\
 
%    General info

%60G25 (1973-now) Prediction theory (aspects of stochastic processes)
%62M20 (1973-now) Inference from stochastic processes and prediction

%60D05 (1973-now) Geometric probability and stochastic geometry

%60G10 (1973-now) Stationary stochastic processes
%60G15 (1973-now) Gaussian processes
%60G52 (2000-now) Stable stochastic processes
%60G60 (1980-now) Random fields

%-----------------------------------------------------------------------------------------------------------------------------------------------------------------------------------------------------------------------
\section{Introduction} \label{sectIntro}

Let $Y:\Omega\to\R$ be a square integrable random variable defined on a probability space $(\Omega,{\cal F}, \pr)$, and let ${\cal G}\subset {\cal F}$ be a sub--$\sigma$--algebra generated by a family of random variables $\{Z_j\}$  which are observable.
The classical $L^2$--theory of prediction of random variables states that the best unbiased predictor of  $Y$ with respect to  ${\cal G}$ is given by the conditional expectation $\E(Y | {\cal G}) $ which 
is an orthogonal $L^2$--projection of $Y$ onto the space of ${\cal G}$--measurable square integrable random variables. But as far as $Y$ has no finite moments, no unified widely accepted prediction theory exists, to the best of our knowledge. Our paper is an attempt to create such theory which also applies to the finite variance case. Its main idea is the following.
Let $u\in \R$ be an excursion level chosen according to a finite measure $m(\cdot)$ on $\R$. For any two random variables $Y_1, Y_2:\Omega\to\R$  introduce the quantity
$$
E_m(Y_1,Y_2):=\E \int_\R \ind( \{ Y_1>u  \} \bigtriangleup   \{ Y_2>u  \} ) \, m(du)=\int_\R \pr( \{ Y_1>u  \} \bigtriangleup   \{ Y_2>u  \} ) \, m(du),
$$
which is (by Fubini's theorem) a  $m$-weighted average probability of symmetric difference of excursions of $Y_1$ and $Y_2$ over $u\in\R$. Then, we say that $Z$ is a prediction of a random variable $Y$ onto the $\sigma$--algebra ${\cal G}$ introduced above if 
$$
Z=\mbox{argmin}_{Y_0}  E_m(Y,Y_0),
$$
whenever this minimum (taken over all ${\cal G}$-measurable random variables $Y_0$) exists and is unique. Sometimes it is also plausible to add more constraints to the geometry of our projection space saying that, additionally to ${\cal G}$-measurability, $Y_0\eqd Y$ (here, $\eqd$ means the equality in law) or that $Y_0$ is a linear combination of  $\{Z_j\}$. Apparently, the above solution $Z$, its existence and uniqueness may heavily depend on the choice of measure $m(\cdot)$. A natural candidate for this would be the distribution of $Y$ as we explain it later.	
As it is shown in Theorem \ref{thm:metric}, $E_m$ is a metric on the space of random variables whenever the distribution function of $m(\cdot)$ is strictly increasing.

The intuition behind the use of the new metric is the following. Assume that a stationary heavy--tailed time series $\{Y_t,t\in\R\}$ is observed at locations $t_1,\ldots, t_n$ in a compact window $W\subset \R$. As proposed in \cite{Dasetal22}, the linear predictor 
$\widehat{Y}_t=\sum\limits_{j=1}^n \lambda_j Y_{t_j}$, $t\not\in \{ t_1,\ldots, t_n \}$, is a minimizer of the functional
$$
 \int_\R \E \left[v_{1}\left(A_{Y}(u) \Delta A_{\widehat{Y}}(u)\right)\right] \, m(du)=\int\limits_W E_m(Y_t,\widehat{Y}_t) \, dt$$
with respect to the choice of weights $\lambda_1,\ldots,\lambda_n$ subject to the constraint $\widehat{Y}_t\eqd Y_t$. The above equality holds by Fubini's theorem, whereas the left hand side term is the mean length of the symmetric difference of excursion sets
$A_{Y}(u) :=\{ t\in W: Y_t>u  \}$ and $A_{\widehat{Y}}(u):=\{ t\in W: \widehat Y_t>u  \}$ averaged over the levels $u\in\R$ picked up according to the measure $m(\cdot)$. The term $\E \left[v_{1}\left(A_{Y}(u) \Delta A_{\widehat{Y}}(u)\right)\right] $ is called {\it error-in-measure} and quantifies the prediction error measured by the symmetric difference of excursions. Here and in what follows,  $v_1(\cdot)$ is the Lebesgue measure on $\R$. In view of the said above, $E_m$ will be named the { \it excursion metric}. In the previous literature, the minimization of probability metrics was used mainly in context of optimal mass transportation and parameter inference, see e.g. \cite{rachev2013methods}.

The paper is organised as follows: in Section \ref{sectExcursMetric} the properties of $E_m$ are studied. It is shown that $E_m$ coincides with the so--called {\it separation (pseudo) metric }\cite{Taylor84} whenever $m$ is a probability measure.
The maximal value attained by $E_m$ with respect to the choice of measure $m$ as well as implications of the choice $m=\pr_Y$ are given in Section \ref{sectOptChoice}.
Restricted to the space of random variables $Y$ with the same absolutely continuous distribution $F$,  the metric $E_{F}$ turns to be distribution--free depending only on bivariate copulas. We call this metric (in analogy to Gini coefficient from econometrics \cite{Gast72,Yitz13}) a {\it Gini metric}. It properties are investigated within the same section.
The excursion metric with $ m=\pr_Y$ is applied to the prediction of random variable $Y$  in Section \ref{sectPredict}. There, we give three possible forms of minimization problems 
leading to such forecasting. Existence of the solution and consistency of the predictor are discussed in Section \ref{sect:Exist}.  A special case of extrapolating heavy-tailed time series is considered in Section \ref{sectExtrapol}.  Numerical examples predicting Gaussian, $\alpha$--stable   and  autoregressive heavy--tailed stationary time series follow in Section \ref{sect:Num}.

%Let $X=\{X(t),t\in \R^d\}$ be real-valued ID random field. Denote by  $A_X(u):=\{t\in W,\, X(t)>u\}$ an excursion set of level $u\in \R.$ Consider the general problem of prediction of $X(t_0)$ by observations at points $T=\{t_1,\ldots,t_n\}$ via linear interpolator $\widehat{X}(t_0)=\sum_{j=1}^n \lambda_j X(t_j),$ where $\lambda=(\lambda_1,\ldots,\lambda_n)\in \R^n.$ Let $u_j,j=1,\ldots,k$ be real numbers. Interpolation coefficients $\lambda$ are found via minimization of the following functional
%$$F(\lambda)=\sum_{j=1}^k \E \nu_d(A_X(u_j)\Delta A_{\widehat{X}}(u_j))=\int_{W}\sum_{j=1}^k\Delta_{X(t),\widehat{X}(t)}(u_j,u_j)dt,$$
%where
%\begin{align*}
%    \label{delta1}
%\Delta_{Y_1,Y_2}(\mathbf{v})&=\pr(Y_1>v_1,Y_1>v_2)+\pr(Y_2>v_1,Y_2>v_2)\\
%&-\pr(Y_1>v_1,Y_2>v_2)-\pr(Y_1>v_2,Y_2>v_1),\, \mathbf{v}=(v_1,v_2)\in \R^2.
%\end{align*}

%-----------------------------------------------------------------------------------------------------------------------------------------------------------------------------------------------------------------------
\section{Excursion metric and its properties} \label{sectExcursMetric}

Let $L^0(\Omega, {\cal F}, \pr)$ be the set of all real--valued random variables define on a probability space   $(\Omega, {\cal F}, \pr)$.

We introduce the {\it excursion (pseudo)metric} mentioned in Section \ref{sectIntro} in a slightly different (but equivalent) form:
\begin{definition}
\label{Em:def}
Let $m$ be a finite non-negative measure on $\R,$ and $Y_1, Y_2 \in L^0(\Omega, {\cal F}, \pr) $. The excursion metric $E_m$ is given by 
\begin{equation}
\label{Em:def:eq}
    E_m(Y_1,Y_2):=\int_\R \left(\pr(Y_1>u)+\pr(Y_2>u)
-2\pr(Y_1>u,Y_2>u)\right) m(du).
\end{equation}
\end{definition}
In order to understand when functional $E_m$ is indeed a metric on the space $L^0(\Omega,{\cal F}, \pr)$, we present several equivalent forms of \eqref{Em:def:eq}.

In the sequel, we will need the distribution function of measure $m$ given by 
 $$F_U(x):=\int_{-\infty}^xm(du), x\in\R,\quad  F_U(x-)=\lim_{y\to x-}F_U(y),$$
and the notation
$$\Delta_{Y_1,Y_2}(u):=\pr(Y_1>u)+\pr(Y_2>u)
-2\pr(Y_1>u,Y_2>u).$$
It clearly holds  $E_m(Y_1,Y_2)=\int_\R \Delta_{Y_1,Y_2}(u) m(du).$
Denote by $Y_1\vee Y_2$ ($Y_1\wedge Y_2$) the maximum (minimum, resp.) of the random variables $Y_1$  and $Y_2$.
\begin{remark}
Let $F_1$ and $F_2$ be the distribution functions of $Y_1$ and $Y_2,$ respectively, and $C$ be the copula of $(Y_1,Y_2).$ Then  one writes
\begin{align}
\label{delta:eq0}\Delta_{Y_1,Y_2}(u)&=\pr(Y_1\vee Y_2> u)-\pr(Y_1\wedge Y_2> u)
=\pr(Y_1\wedge Y_2\leq u)-\pr(Y_1\vee Y_2\leq  u)   \\
%\label{delta:eq1}
&=\pr(Y_1\leq u)+\pr(Y_2\leq u)-2\pr(Y_1\leq u,Y_2\leq u)  \nonumber \\ 
&=F_1(u)+F_2(u)-2\pr(Y_1\vee Y_2\leq u)  =F_1(u)+F_2(u)-2 C(F_1(u),F_2(u)), \label{delta:eq2}
\end{align}
where the last relation follows from Sklar's theorem \cite{Sklar59}.
Moreover, it follows
\begin{align*}
    \Delta_{Y_1,Y_2}(u)&=\frac{1}{2}\pr(Y_1>u)+\frac{1}{2}\pr(Y_2>u)-\pr(Y_1>u,Y_2>u)\\
    &+\frac{1}{2}\pr(Y_1\leq u)+\frac{1}{2}\pr(Y_2\leq u)-\pr(Y_1\leq u,Y_2\leq u)\\
    &=1-\pr((Y_1-u)(Y_2-u)\geq 0)=\pr((Y_1-u)(Y_2-u)< 0)\in [0,1].
\end{align*}
Thus, it follows the relation
$0\leq E_m(Y_1,Y_2)\leq m(\R).$ \end{remark}

Without loss of generality, we may thus divide both sides of the last inequality by $ m(\R)$ and consider $m$ to be a probability measure which yields
$0\leq E_m(Y_1,Y_2)\leq 1.$ 
Let $U$ be a random variable with probability law $m$ which is independent of $Y_1$, $Y_2$. It can be interpreted as a random excursion level which we choose to build the metric $E_m$.  
\begin{lemma}\label{lemmFormEm}
Let $m$ be a probability measure on $\R$ with c.d.f.  $F_U(x)=\int_{-\infty}^x m(dy),$ $x\in\R.$
Then
 \begin{align}
 \label{Em:eq1}
 E_m(Y_1,Y_2)&=\E|F_U(Y_2-)-F_U(Y_1-)|.
 \end{align}
\end{lemma}
\begin{proof}
 We have from relation \eqref{delta:eq0} that 
\begin{align*}
    E_m(Y_1,Y_2)&=\int_{\R}\E \left(\mathbbm{1}\{u<Y_1\vee Y_2\}-\mathbbm{1}\{u<Y_1\wedge Y_2\}\right) m(du)\\
    &=\E \left(F_U(Y_1\vee Y_2-)-F_U(Y_1\wedge Y_2-) \right)=\E|F_U(Y_2-)-F_U(Y_1-)|.
 \end{align*}
\end{proof}
Equation \eqref{Em:eq1} can be interpreted as a probability that $U$ separates $Y_1$ and $Y_2$:
$$
E_m(Y_1,Y_2)=\pr( Y_1\wedge Y_2 \le U<  Y_1\vee Y_2 ).
$$
Seen this way, $E_m$ coincides with the {\it separation (pseudo)metric} introduced by M. Taylor \cite{Taylor84}. The first part of the following corollary is also contained in \cite[Remark 1]{Taylor84}:

\begin{corollary} \label{corEm}
If $F_U$ is continuous then it holds $F_U(x-)=F_U(x),x\in\R,$ and
\begin{equation} E_m(Y_1,Y_2)=\E|F_U(Y_2)-F_U(Y_1)|. \label{eqE_m1}
 \end{equation}
If, in addition, $m$ is absolutely continuous with density $\psi$
then 
\begin{equation*}
 E_m(Y_1,Y_2)=\E\left|\int_{Y_1}^{Y_2}\psi(u)du\right|=\frac{1}{2} \E \left[  |Y_1-Y_2|\int_{-1}^{1}\psi\left(s\frac{|Y_1-Y_2|}{2}+\frac{Y_1+Y_2}{2}\right)ds \right].
 \end{equation*}
\end{corollary}
%\begin{example}
%\begin{enumerate}
%    \item Let measure $\mY_1$ is concentrated on $[-r,r]$ with  constant density $\psi=1.$ Then $F_U(x)=(x+r)_+ \wedge 2r$
%$$E_m(Y_1,Y_2)=\E \left[ (Y_1\vee Y_2+r)_+\wedge 2r-(Y_1\wedge Y_2+r)_+\wedge 2r\right]\leq \E (2r\wedge |Y_1-Y_2|)  $$
%\end{enumerate}
%\end{example}
%We can associate levels $Y_1,\ldots,u_k$ with the measure $ m_U(A)=\sum_{j=1}^k\mathbbm{1}\{u_j\in A\}, A\in \mathcal{B}(\R).$ Then functional $F(\lambda)$ can be rewritten as $F(\lambda)=\int_{W}\int_\R \Delta_{X(t),\widehat{X}(t)}(u,u) m_U(du)dt.$ 

%Let us consider a more general situation.
%\begin{proposition}
%\label{prop1}
%Let $ m_U$ be a finite measure on $\R$ with $F_U(x)=\int_{-\infty}^xm(dy),$ $x\in\R.$
%ThenY_3

%\end{proposition}

The next theorem was also proven (under different assumptions) in \cite[Theorem 2]{Taylor84}:

\begin{theorem} \label{thm:metric}
Let $X_S$ be the space of random variables with support $S\subseteq\R.$ If $F_U$ is strictly increasing on $S,$ then $E_m$ is a metric on $X_S\times X_S.$
\end{theorem}
\begin{proof}
The symmetry of $E_m$ is trivial. The triangle inequality  follows for arbitrary $Y_1,Y_2,Y_3\in X_S$ from 
\begin{align*}
    &E_m(Y_1,Y_2)= \E|F_U(Y_2-)-F_U(Y_1-)|\\
    &\leq \E|F_U(Y_2-)-F_U(Y_3-)|+\E|F_U(Y_3-)-F_U(Y_1-)|=E_m(Y_1,Y_3)+E_m(Y_3,Y_2).
\end{align*}
Let for some $Y_1,Y_2\in X_S:$ $\E|F_U(Y_2-)-F_U(Y_1-)|=0,$ then $F_U(Y_1-)=F_U(Y_2-)$ a.s. Since $F_U$ is strictly increasing on $S,$ $\pr(Y_1=Y_2)=1$. Thus, $E_m$ is a metric.
\end{proof}
It can be easily shown  that $E_m$ given in \eqref{eqE_m1}  (with a specific choice of $U:(0,1) \to \R$ being a homeomorphism) coincides with the metric $d_{h,p}$ for $p=1$, $h=U$ from \cite{Taylor85} which metrizes the weak convergence in the space of distribution functions.

%-----------------------------------------------------------------------------------------------------------------------------------------------------------------------------------------------------------------------
\section{Optimal choice of a weighting measure $m$} \label{sectOptChoice}

In this section, we assume that random variables $Y_1$ and $Y_2$ are absolutely continuously distributed with support $S\subseteq\R$, cumulative distribution functions (c.d.f.'s) $F_{1}$, $F_2$ and copula $C(\cdot,\cdot)$.
 Let $D$ be the space of all probability measures on $\R$. 
\begin{theorem}
The maximum
 $$\max_{m\in D} E_m(Y_1,Y_2)$$
 is attained at  a measure $m$ with c.d.f. $F_U^*(x)=\mathbbm{1}\{u^*\leq x\},$ where $$u^*=\arg\max_{u\in\R}(F_1(u)+F_2(u)-2 C(F_1(u),F_2(u))).$$

If additionally $Y_1\stackrel{d}{=}Y_2$ then 
\begin{equation}\label{eq:maxEm}
\max_{m\in D} E_m(Y_1,Y_2)=2\max_{x\in[0,1]}(x-C(x,x)),
\end{equation}
whereas $u^*=F_1^{-1}(x^*)$ and  $x^*=\arg\max_{x\in[0,1]} (x-C(x,x)).$
\end{theorem}
\begin{proof}
Recall from \eqref{delta:eq2} that
    $\Delta_{Y_1,Y_2}(u)=F_1(u)+F_2(u)-2 C(F_1(u),F_2(u))\in [0,1].$
Therefore, there exists, not necessarily unique, 
$$u^*=\arg \max_{u\in \R} \left[F_1(u)+F_2(u)-2 C(F_1(u),F_2(u))\right].$$
Then 
$$
E_m(Y_1,Y_2)=\int_\R \Delta_{Y_1,Y_2}(u)  \, m(du)\leq F_1(u^*)+F_2(u^*)-2C(F_1(u^*),F_2(u^*))=\int_\R \Delta_{Y_1,Y_2}(u) \delta_{u^*}(du).
$$
If $Y_1\stackrel{d}{=}Y_2$ then $F_1=F_2$ and hence 
$$2 \max_{u\in \R} \left[F_1(u)-C(F_1(u),F_1(u))\right]= 2 \max_{x\in [0,1]} \left[x-C(x,x)\right],$$
since $F_1$ is non-decreasing.
\end{proof}

\begin{corollary} \label{corMaxValofE_m}
Let the random vector $(Y_1,Y_2)$ have a density function $p:\R^2\to\R_+$ with $Y_1\stackrel{d}{=}Y_2$.  Then $u^*$ satisfies the equation
$$\int_{u^*}^{+\infty} p(y,u^*)dy=\int_{-\infty}^{u^*} p(u^*,y)dy.$$
If  $p$ is additionally unimodal  and symmetric around its mode  $(\mu,\mu)$ then $u^*=\mu,$  $x^*=0.5,$
and $$\max_{m\in D} E_m(Y_1,Y_2)=1-2C(0.5,0.5).$$ 
\end{corollary}
\begin{proof} It holds
\begin{align*}
    &\max_{x\in[0,1]}(x-C(x,x))=\max_{u\in\R}(F(u)-C(F(u),F(u)))\\
    &=\max_{u\in \R}\left(\int_\R\int_{-\infty}^u p(y_1,y_2)dy_1dy_2-\int_{-\infty}^u\int_{-\infty}^u p(y_1,y_2)dy_1dy_2\right)\\
    &=\max_{u\in \R}\int_u^{+\infty}\int_{-\infty}^u p(y_1,y_2)dy_1dy_2.
\end{align*}
The maximum is reached on an extremal point $u$ such that
$$\frac{d}{du}\int_u^{+\infty}\int_{-\infty}^u p(y_1,y_2)dy_1dy_2=0,$$ or 
$$\int_u^{+\infty} p(y,u)dy=\int_{-\infty}^{u} p(u,y)dy.$$

Now let $p$ be unimodal and symmetric around its mode $(\mu,\mu),$ then
\begin{align*}
    &\int_\mu^{+\infty} p(y,\mu)dy=\int_{-\infty}^\mu p(y,\mu)dy=\frac{1}{2} \int_\R p(y,\mu)dy,\\
    &\int_{-\infty}^\mu p(\mu,y)dy=\int_\mu^{+\infty} p(\mu,y)dy=\frac{1}{2} \int_\R p(\mu,y)dy=\frac{1}{2} \int_\R p(y,\mu)dy.
\end{align*}
Thus, the maximum is reached at $u^*=\mu.$
\end{proof}

\begin{remark}
Although the maximum of $E_m$ is reached on $m=\delta_{\{u^*\}},$
 $E_{\delta_{\{u^*\}}}$ is not a metric but just a pseudo metric.

 % In such case, the supremum of $E_m$ over strictly increasing $F_U$ on $S$  remains to be  $2\max_{x\in [0,1]}(x-C(x,x)).$ 
%It is approached by $M_h(x)=F_1\left(\frac{x-u^*}{h}\right),x\in \R.$ Indeed, under condition $Y_1\stackrel{d}{=}Y_2,$
%\begin{align*}
   % \E_{m_h} (Y_1,Y_2)&=2\int_{\R}\left[F_1(u)-C(F_1(u),F_1(u))\right]d M_h(u)\\
   % &=2\int_{\R}\left[F_1(u)-C(F_1(u),F_1(u))\right]d F_1\left(\frac{x-u^*}{h}\right)\\
   % &=2\int_{0}^1\left[F_1(u^*+hF_1^{-1}(y))-C(F_1(u^*+hF_1^{-1}(y)),F_1(u^*+hF_1^{-1}(y)))\right]d y\\
   % &\to 2\left[F_1(u^*)-C(F_1(u^*),F_1(u^*))\right]\quad \text{ as }h\to 0.
%\end{align*}
\end{remark}

Which choice of  $m$ is preferable to keep $E_m$ a metric which is relatively easy to compute, infer and interpret? 
 We know that $F_U$ should be strictly increasing on support $S.$  If we take $F_U=F_1$ then  it follows from the proof of Lemma \ref{lemmFormEm} that 
\begin{align}
E_{F_1}(Y_1,Y_2) &=    \E F_1 (Y_1\vee Y_2)-\E F_1 (Y_1\wedge Y_2) \nonumber \\
    &=2\E F_1 (Y_1\vee Y_2)-\E F_1 (Y_1)-\E F_1 (Y_2)=2\E F_1 (Y_1\vee Y_2)-\E F_1 (Y_2)-\frac{1}{2}, \label{EF1:eq1}
\end{align}
since it holds $ F_1 (Y_1)\sim U(0,1)$ with expected value $1/2$. $E_{F_1}$ is a metric on the space of all random variables with absolutely continuous distributions on support $S$. 

 If additionally $Y_1\stackrel{d}{=}Y_2,$ we get 
$$E_{F_1}(Y_1,Y_2)=2\E F_1(Y_1\vee Y_2)-1=2\E (F_1(Y_1)\vee F_1(Y_2))-1$$
from equation \eqref{EF1:eq1} since $F_1$ is strictly increasing. It follows from \cite[p. 68]{DurSem16} that $$C(x,x)=\pr( F_1(Y_1) \vee F_1(Y_2) \le x  ), \quad x\in[0,1].$$ 
Using relation  \eqref{delta:eq2} one writes after the substitution $x=F_1(u)$ that
 \begin{equation}\label{eq:E_mExp}
E_{F_1}(Y_1,Y_2)=2\int_{\R}\left[F_1(u)-C(F_1(u),F_1(u))\right]d F_1(u)=2\int_{0}^1\left[x-C(x,x)\right]dx=1-2 \int_0^1 C(x,x) \, dx.
\end{equation}
Since the term
$ 
2\int_{0}^1\left[x-C(x,x)\right]dx
$ is equal to the {\it Gini coefficient} of the {\it  Lorenz curve} $$\{  (x, C(x,x)), \; x\in[0,1]  \}$$ in case of a convex $ \{  C(x,x), \; x\in[0,1]  \}$, we come to the following definition:

\begin{definition}
Let $L_{F_1}$ be a space of random variables  with absolutely continuous c.d.f. $F_1$. The metric $G=E_{F_1}$ given by
\begin{equation*}\label{eq:G}
G(Y_1,Y_2)=1-2 \int_0^1 C(x,x) \, dx, \quad Y_1, Y_2\in L_{F_1},
\end{equation*}
where $C$ is the copula of  $(Y_1, Y_2)$, is called a Gini metric on $L_{F_1}$. 
\end{definition}
By definition, the Gini metric is  distribution--free: it takes only the dependence structure between $Y_1$ and $Y_2$ into account, but not the marginal distribution of $Y_1,Y_2$.

\begin{remark}
In case $Y_1\stackrel{d}{=}Y_2$, the maximum value \eqref{eq:maxEm} equals the $L^\infty([0,1])$--distance between the diagonals of $C$ and of the upper Fr\'{e}chet--Hoeffding bound $M_2(x,y)=\min\{x,y\},$ $x,y\in[0,1]$, cf. e.g. \cite[Theorem 1.7.3]{DurSem16}.  Namely, 
\begin{equation}\label{eq:DiagLinf}
\sup_{m\in D}  E_{m}(Y_1,Y_2)=2\| M_2(x,x) -C(x,x)    \|_{L^\infty([0,1])}.\end{equation}
Similarly, relation \eqref{eq:E_mExp} yields $G(Y_1,Y_2)=2\| M_2(x,x)-C(x,x)\|_{L^1([0,1])}$. In other words, the Gini metric $G$ measures the $L^1$--deviation of the diagonal of the copula $C$ of  $(Y_1, Y_2)$ to the diagonal of the comonotonicity copula $M_2$.
\end{remark}
%Integrating $ \int_0^1 C(x,x)
 %\, dx$ by parts we get 
%E_{F_1}(Y_1,Y_2)=2\int_0^1x\, dC(x,x)-1= 2 \E [F_1(Y_1) \vee F_1(Y_2)] -1 .

\begin{lemma}\label{lemm:BoundsG}
It holds $0\le G(Y_1,Y_2) \le 1/2$,  $Y_1, Y_2\in L_{F_1}$. The upper bound is attained whenever $Y_2=f(Y_1)$ a.s. for some decreasing function $f$ such that
$F_1(x)=1-F_1(f^{-1}(x))$, $x\in \R$.
\end{lemma}
\begin{proof}
Using the Fr\'{e}chet--Hoeffding bounds \cite[Theorem 1.7.3]{DurSem16}
$$
W_2(x,y)\le C(x,y)\le M_2(x,y), \quad x,y\in [0,1]
$$ 
with $W_2(x,y)=\max\{ 0, x+y-1 \}$  being a copula of linearly dependent random variables, one can easily calculate
$$
1/4\le \int_0^1 C(x,x) \, dx \le 1/2.
$$
 Thus, relation  \eqref{eq:E_mExp}  yields the  bounds $0\le G(Y_1,Y_2) \le 1/2$, whereas the upper bound is attained by  \cite[Theorem 2.5.13 (d)]{DurSem16} whenever there exists a decreasing function $f$ such that $Y_2=f(Y_1)$ a.s. Since we assume $Y_1\stackrel{d}{=}Y_2$ here, it can be only the case if 
 $$F_1(x)=\pr(Y_1\le x)= \pr(f(Y_2)\le x)=\pr(Y_2\ge f^{-1}(x))=1- \pr(Y_2< f^{-1}(x))=1-F_1(f^{-1}(x))$$
 for any real $x$. In the latter relation, we used the absolute continuity of the distribution of $Y_2$.
\end{proof}
\begin{example}
The upper bound $1/2$ is attained by Gini metric in Lemma \ref{lemm:BoundsG} if  $Y_1+Y_2=\mu$ a.s. for some $\mu\in\R$  and  the distribution of $Y_1$ is symmetric about $\mu$:
$F_1(x)=1-F(\mu-x)$, $x\in \R$. To see this, just take $f(x)=\mu-x$ in  Lemma \ref{lemm:BoundsG}.
\end{example}
\begin{example}\label{ex:GiniIndependent}
The Gini distance between stochastically independent  random variables $Y_1, Y_2\in L_{F_1}$  is equal to $1/3$, since in this case $C(x,y)=xy$, $x,y\in[0,1],$ and thus 
$G(Y_1,Y_2) = 1-2\int_0^1 x^2 \,dx=  1/3$.
\end{example}

To summarize, the choice $m=\pr_{Y_1}$ for an absolutely continuous law $\pr_{Y_1}$ seems  natural to make $E_m$ a metric. Gini metric $G$ is easy to calculate and distribution--free (depending only on the copula $C$ of $(Y_1,Y_2)$, cf. relation \eqref{eq:E_mExp})  when $\pr_{Y_1}=\pr_{Y_2}$. Informally speaking, the excursion levels are chosen here according to the same law as  $\pr_{Y_1}=\pr_{Y_2}$ which makes the corresponding excursion sets non--empty (with positive probability) and representative for $Y_1, Y_2$.

%-----------------------------------------------------------------------------------------------------------------------------------------------------------------------------------------------------------------------
\section{Prediction of random variables}\label{sectPredict}

In this section, we discuss the prediction of a value of random variable $X$ with continuous distribution function $F_X$ based on the set $\mathcal{X}_n:=(X_1,\ldots,X_n)$ of realizations of $X$  via the excursion metric metric $E_{F_X}$. Namely, we propose a predictor $\widehat{X}_\lambda:=g(\lambda,\mathcal{X}_n),$ where $\lambda=(\lambda_1,\ldots, \lambda_d) \in \Lambda\subset\R^d$ is deterministic  and $g:\R^n\times \Lambda\to \R$, $d,n\in\N$, is a continuous measurable function such that the excursion metric is minimal:

\begin{equation}\label{eq:E_PredictFormal}
E_{F_X}(X, \widehat{X}_\lambda) \to \min_{\lambda\in \Lambda}.
\end{equation}

Here, the set of admissible parameters $\Lambda$ as well as the analytic form of $g(\lambda,\mathcal{X}_n)$ depend on  the law $\pr_X$. For instance, the choice $g(\lambda,\mathcal{X}_n)=\sum_{j=1}^n \lambda_j X_j$ makes sense for infinitely divisible laws of $X$, whereas  $g(\lambda,\mathcal{X}_n)=\max_{j=1,\ldots,n} \lambda_j X_j$ might be a better choice for max--stable $X$. In both cases, we assume $d=n$. The set $\Lambda$ may incorporate additional constraints onto $\pr_{\widehat X}$, for instance, $\widehat{X}_\lambda \eqd X$. 
Since $g$ and $F_X$ are continuous, the constraint $\widehat{X}_\lambda\stackrel{d}{=}X$ is equivalent to $F_X(\widehat{X}_\lambda)\stackrel{d}{=}F_X(X)\stackrel{d}{=}U,$ where $U\sim U(0,1).$ Under additional assumptions onto the joint probability law of  $(X, \mathcal{X}_n)$ and onto $g$, the set  $\Lambda_g:=\{\lambda\in \R^d :F_X(g(\lambda,\mathcal{X}_n))\stackrel{d}{=}U\}$ is a manifold in $\R^d$. Unfortunately, the analytic form of $\Lambda_g$ can be found only in specific cases when the pre-knowledge of the distribution of $(X,\mathcal{X}_n)$ is available, such as in the Gaussian or $\alpha$--stable case, cf. \cite{Dasetal22}.

Should our prediction be law-preserving (i.e.,  $\widehat{X}_\lambda \eqd X$), it holds  $E_{F_X}=G$, and   the optimization problem 
$G(X, \widehat{X}_\lambda) \to \min_{\lambda\in \Lambda}$
with $\Lambda=\Lambda_g$ rewrites using \eqref{eq:E_mExp} as
\begin{equation}\label{eq:E_PredictFormal_Copula}
\int_0^1 C _{X, \widehat{X}_\lambda}(x,x) \,dx  \to \max_{\lambda\in \Lambda_g},
\end{equation}
where $C _{X, \widehat{X}_\lambda}$ is the copula of $(X, \widehat{X}_\lambda)$. In order to avoid a tricky statistic assessment of copulas, we use however the following  forms of prediction which are motivated by  \eqref{EF1:eq1}   and  require only expectations to be inferred into: 

 %\begin{equation}\label{eq:E_mExp}

%\end{equation}

 %\begin{equation}\label{eq:E_mExp}

%\end{equation}

\begin{definition}
\label{def:excurs:pred}
The excursion predictor $\widehat{X}_\lambda$ is given by  $\widehat{X}_\lambda=g(\hat{\lambda},\mathcal{X}_n)$,  where 
\begin{equation}
    \label{unconst:eq1}
    \hat{\lambda}:=\arg \min_{\lambda\in \Lambda}\left[2\E F_X (X\vee \widehat{X}_\lambda)-\E F_X (\widehat{X}_\lambda)\right]
\end{equation}
in general case, or
 \begin{equation}
    \label{const:eq2}  
    \hat{\lambda}:=\arg \min_{\lambda\in \Lambda_g}\left\{\E F_X (X\vee \widehat{X}_\lambda)\right\}
\end{equation}
in case of law-preserving prediction $\widehat{X}_\lambda\stackrel{d}{=}X$.
\end{definition}

Let us consider the law-preserving case in more detail. If the analytic form of $ \Lambda$ is given explicitly but $ \Lambda_g$ is hardly available, we modify the minimization functional in \eqref{const:eq2} by adding a term which penalizes a  difference between the law of $Y_1=F_X(\widehat{X}_\lambda)$ and $Y_2\sim U(0,1)$:

\begin{equation}
    \label{const:eq3} 
    \hat{\lambda}:=\arg \min_{\lambda\in \Lambda}\left\{2 \E F_X (X\vee \widehat{X}_\lambda)    - \E F_X (\widehat{X}_\lambda) +\gamma \rho^2(F_{Y_1},F_{Y_2})\right\},
\end{equation}
where $\gamma>0$ is a penalty weight and $\rho$ is an arbitrary (but handy) metric  on the space of continuous distribution functions of random variables.
 For simplicity reasons, we use the {\it 2-Wasserstein distance }
    $$\rho(F_1,F_2)=\left(\int_0^1 [F^{-1}_1(x)-F^{-1}_2(x) ]^2dx\right)^{1/2}$$
   between two c.d.f.'s $F_1$ and $F_2$ with quantile functions  $F^{-1}_1$ and  $F^{-1}_2$, respectively.
    
    In the case $Y_2\sim U(0,1)$ we have  $F_{Y_2}^{-1}(x)=x$,  $x\in[0,1]$. Hence, the squared 2-Wasserstein distance equals
    \begin{equation}
    \label{eq:Wasserstn} 
    \rho^2(F_{Y_1},F_{Y_2})=\int_0^1 x^2 dx +\int_0^1 y^2dF_{Y_1}(y)-\int_{0}^1 y d F_{Y_1}^2(y)=\frac{1}{3}+\E Y_1^2-\E [Y_1\vee Y],
    \end{equation}
     where $Y$ is an independent copy of $Y_1.$ The latter relation holds since $\pr(Y_1\vee Y\le y)=  F_{Y_1}^2(y)$, $y\in \R$.
Due to $Y_1=F_X(\widehat{X}_\lambda)\in [0,1]$ a.s., it holds  $\E [Y_1\vee Y]\ge \E Y_1^2$. Thus, the minimization problem \eqref{const:eq3}  rewrites in an equivalent form:
\begin{equation}
    \label{eq:FinalPreidct} 
    \hat{\lambda}:=\arg \min_{\lambda\in \Lambda}\left\{2 \E F_X (X\vee \widehat{X}_\lambda)    - \E F_X (\widehat{X}_\lambda)+\gamma \left[  \E F_X^2(\widehat{X}_\lambda)- \E [F_X(\widehat{X}_\lambda)\vee Y]   \right]       \right\},
\end{equation}    
where $Y$ is an independent copy of $F_X(\widehat{X}_\lambda)$.
    Compared with formulation \eqref{const:eq2}, the new prediction method   \eqref{eq:FinalPreidct} does not require an explicit knowledge of $\Lambda_g$, but it realizes the constraint 
    $\widehat{X}_\lambda\stackrel{d}{=}X$ only in approximation form: $\rho(F_X, F_{\widehat{X}_\lambda})\le \varepsilon$ for some small $\varepsilon>0$.
    
    Sometimes it is advantageous to use the integration by parts in  \eqref{eq:Wasserstn} and write
   \begin{equation}
    \label{eq:Wasser1} 
    \rho^2(F_{Y_1},F_{Y_2})=\frac{1}{3}+ \int_0^1 F_{Y_1}(y) \left[ F_{Y_1}(y)-2y\right]\, dy
  \end{equation}
  which allows for an equivalent reformulation
   \begin{equation}
   \label{eq:FinalPreidct1} 
    \hat{\lambda}:=\arg \min_{\lambda\in \Lambda}\left\{2 \E F_X (X\vee \widehat{X}_\lambda)    - \E F_X (\widehat{X}_\lambda) +\gamma  \int_0^1 F_{Y_1}(y) \left[ F_{Y_1}(y)-2y\right]\, dy     \right\}
\end{equation}    
of the problem     \eqref{eq:FinalPreidct}.

   % Another relation of is 
    %$W^2_2(F_1,F_2)=\frac{1}{3}+\E Y^2-2\int_0^1 y F_2^{-1}(y)dy.$

%hier aufgehoert

%For instance, $\rho$ can be defined via $L^p(\psi,\R)$ norm as
%$$\rho(F_1,F_2):=\|F_1-F_2\|_{L^p(\psi,\R)}=\left(\int_\R |F_1(x)-F_2(x)|^p\psi(dx) \right)^{1/p},\quad F_1,F_2\in C_{0,1}(\R).$$ 
%The special cases are 
%\begin{itemize}
   % \item Kolmogorov-Smirnov statistics ($\psi(dx)=dx,p=\infty$). For instance, $\rho_\infty(F_n,F)=\sup _{x\in \R}|F_{n}(x)-F(x)|,$ where $F_n$ is a empirical distribution function.
    %\item Wasserstein distance ($\psi(dx)=dx,p=1$): $\rho_1(F_1,F_2)=\int_\R |F_1(x)-F_2(x)|dx.$
    %\item
% \end{itemize}

To summarize, prediction approach \eqref{unconst:eq1} will be used for unconstrained prediction of a random variable $X$ with an absolutely continuous c.d.f. $F_X$ based on its realizations $\mathcal{X}_n=(X_1,\ldots,X_n)$. For the law--preserving prediction, approaches  \eqref{const:eq2}, \eqref{eq:FinalPreidct} or \eqref{eq:FinalPreidct1}  will be used depending on whether the restrained parameter set $\Lambda_g$ is given explicitly or not.

%----------------------------------------------------------------------------------
\section{Existence of a solution}\label{sect:Exist}

A solution to the above optimization problems exists on compact parametric sets due to the continuity of the corresponding target functionals: 

\begin{theorem}\label{thm:ExistSol}
Let the joint distribution of the random vector $(X,\mathcal{X}_n)$ be absolutely continuous with respect to the Lebesgue measure on $\R^{n+1}$. Introduce the following assumptions:
\begin{enumerate}
\item[\rm (i)] $\Lambda$ (or $\Lambda_g$, respectively) is a compact in $\R^n$.
\item[\rm (ii)] The copula diagonal $C _{X, \widehat{X}_\lambda}(x,x)$   of $(X, \widehat{X}_\lambda)$ is continuous on $\Lambda$ (or $\Lambda_g$, respectively)  uniformly w.r.t. $x\in[0,1]$.
\item[\rm (iii)] For each $\lambda\in\Lambda$, $\widehat{X}_\lambda$ has an absolutely continuous distribution with density   $p_{\widehat{X}_\lambda}$   such that the map $p_{\widehat{X}_\lambda}:\Lambda\to L^1(\R)$ is continuous  on $\Lambda$ w.r.t. the $L^1$--norm. 
\end{enumerate}
If the conditions {\rm (i)-(ii)} hold then there exists a solution to the problem \eqref{const:eq2}. If the conditions {\rm (i)-(iii)} hold then there exists a solution to the problems \eqref{unconst:eq1} and \eqref{eq:FinalPreidct1}. 
\end{theorem} 
\begin{proof}
To show the existence of a solution, it is sufficient to assume (i) and show that the target functional $\Phi(\lambda)$ to be maximized or minimized is continuous on $\Lambda_g$ or $\Lambda$.

In case of the problem \eqref{const:eq2}, we have 
$$
\Phi(\lambda)=\int_0^1 C _{X, \widehat{X}_\lambda}(x,x) \,dx, \quad \lambda\in\Lambda_g
$$
with regard to \eqref{eq:E_PredictFormal_Copula}. Condition (ii) ensures the continuity of $\Phi$ on $\Lambda_g$ in view of the corresponding theorem for the continuity of integrals with parameters.

For the problem \eqref{unconst:eq1}, we have 
$$
\Phi(\lambda)=2-2\int_0^1 C _{X, \widehat{X}_\lambda}(x,x) \,dx  -\E F_X (\widehat{X}_\lambda), \quad \lambda\in\Lambda 
$$
together with
$$
\E F_X (\widehat{X}_\lambda)=\int_0^1 \pr( F_X (\widehat{X}_\lambda)>y) dy=1- \int_0^1 F_{\widehat{X}_\lambda} ( F_X^{-1}(y)) dy,
$$
where $ F_{\widehat{X}_\lambda}$ is the c.d.f. of the predictor $\widehat{X}_\lambda$.
The latter relation holds since the c.d.f. $F_X$ is strictly increasing on  $S=\mbox{supp} (X)$. Due to the absolute continuity of the distribution of 
$\widehat{X}_\lambda$, the function $F_{\widehat{X}_\lambda} ( F_X^{-1}(y))$ is continuous on $[0,1]$ for each $\lambda\in\Lambda$.
Moreover, for any sequence $\{ \lambda_k\} \subset \Lambda$ with $\lambda_k\to\lambda_0\in\Lambda$ as $k\to\infty$  we have
$$
\sup_{x\in S} \left| F_{\widehat{X}_{\lambda_k}}(x)-  F_{\widehat{X}_{\lambda_0}}(x) \right|=\sup_{x\in \R} \left|  \int_{-\infty}^x \left[ p_{\widehat{X}_{\lambda_k}}(y)-  p_{\widehat{X}_{\lambda_0}}(y)   \right] dy \right|\le \int_\R  \left| p_{\widehat{X}_{\lambda_k}}(y)-  p_{\widehat{X}_{\lambda_0}}(y) \right| dy\to 0
$$
as $k\to \infty$ by assumption (iii) which means the continuity of $ F_{\widehat{X}_\lambda} ( F_X^{-1}(y))$ on $ \Lambda$ uniformly w.r.t. $y\in[0,1]$.
The application of the theorem on the continuity of integrals with parameters finishes the proof.

In the problem \eqref{eq:FinalPreidct1},
the target functional rewrites 
$$
\Phi(\lambda)=2-2\int_0^1 C _{X, \widehat{X}_\lambda}(x,x) \,dx  -\E F_X (\widehat{X}_\lambda)+\gamma  \int_0^1 F_{Y_1}(y) \left[ 2y-F_{Y_1}(y)\right]\, dy    , \quad \lambda\in\Lambda. 
$$
Similarly to the previous case, it is not difficult to show that the integrand  $ F_{Y_1}(y) \left[ 2y-F_{Y_1}(y)\right]$  is  continuous on $\Lambda$ uniformly w.r.t. $y\in[0,1]$ provided that condition (iii) holds true.
\end{proof}

\begin{remark}
Condition (iii) of Theorem \ref{thm:ExistSol} means that for any sequence $\{ \lambda_k\} \subset \Lambda$ with $\lambda_k\to\lambda_0\in\Lambda$ as $k\to\infty$  
$$d_{TV}(\widehat{X}_{\lambda_k},\widehat{X}_{\lambda_0})=\frac{1}{2}  \int_\R  \left| p_{\widehat{X}_{\lambda_k}}(y)-  p_{\widehat{X}_{\lambda_0}}(y) \right| dy
=\frac{1}{2}   \left\|  p_{\widehat{X}_{\lambda_k}}-  p_{\widehat{X}_{\lambda_0}} \right\|_1\to 0, \quad k\to\infty,
$$
where   $d_{TV}$ is the total variation distance and $\|  \cdot \|_1$ is the norm in $L^1(\R)$. It implies that  $\widehat{X}_{\lambda_k} \to \widehat{X}_{\lambda_0}$ in total variation as $k\to\infty$. 
\end{remark}
Let us give some examples of the compacts $\Lambda_g\subset \R^n$. In what follows, the random vector $(X,\mathcal{X}_n)$ will have a joint $\alpha$--stable distribution for some 
 $\alpha\in(0,2]$, hence it is natural to consider the linear predictor $\widehat{X}_\lambda=\sum_{j=1}^n \lambda_j X_j$. 
\begin{example}\label{ex:Lambda}

\begin{enumerate}
\item If $(X,\mathcal{X}_n)$ is a Gaussian random vector with marginal distribution $N(\mu, \sigma^2)$ and $\Sigma$ is the covariance matrix of   $\mathcal{X}_n$
then the manifold $\Lambda_g$ is an ellipsoid of dimension $n-1$ given by 
$$
\Lambda_g=\left\{  \lambda=(\lambda_1,\ldots, \lambda_k)\in \R^n:  \lambda^\top \Sigma \lambda=\sigma^2, \quad \sum_{j=1}^n \lambda_j=1 \right\}.
$$
\item If $(X,\mathcal{X}_n)$ is a subgaussian random vector with  stability index $\alpha\in(0,2)$  and i.i.d. standard Gaussian components then it follows from \cite[p. 80-81]{SamTaq94} that 
 $\Lambda_g$ is a unit sphere $S^{n-1}$  in $\R^n$ given by 
$$
\Lambda_g=\left\{  \lambda=(\lambda_1,\ldots, \lambda_k)\in \R^n:   \sum_{j=1}^n \lambda_j^2=1 \right\}.
$$
\item If $(X,\mathcal{X}_n)$ is a symmetric $\alpha$--stable random vector with  stability index $\alpha\in(0,2)$, scale parameter $\sigma=1$ of the marginal distributions   and spectral measure $\Gamma$ of   $\mathcal{X}_n$ then it follows from \cite[p. 73]{SamTaq94} that 
$$
\Lambda_g=\left\{  \lambda=(\lambda_1,\ldots, \lambda_n)\in \R^n:   \int_{S^{n-1}} |\langle  \lambda, s\rangle|^\alpha \, \Gamma(ds)=1 \right\}.
$$
The dominated convergence theorem helps showing that $\Lambda_g$ is a  closed set. 
For $\alpha\in[1,2)$, rewrite the constraint in $\Lambda_g$ as
$$
| \lambda | =  h_K^{-1} (u), \quad u=\lambda/|\lambda|\in S^{n-1}  , 
$$
where  $|\cdot|$ is the Euclidean norm in $\R^n$ and
$$
h_K (u)= \left( \int_{S^{n-1}} |\langle  \lambda, s\rangle|^\alpha \, \Gamma(ds)\right)^{1/\alpha}
$$
is the support function of a convex set $K= \Gamma^{1/\alpha} (S^{n-1}) \E _\alpha [-\eta, \eta]  $ named $L_\alpha$-{\it zonoid}. 
Here $\eta$ is a random vector on $S^{n-1}$ distributed according to $\Gamma (\cdot)/ \Gamma (S^{n-1})$ and $\E_\alpha  [-\eta, \eta]$ is the Firey $\alpha$--expectation of the random segment $ [-\eta, \eta]$, cf. e.g. \cite{Molch09}.
If $K$ is full-dimensional  (which is e.g. the case if $\Gamma$ has a density w.r.t. to the surface area measure on $S^{n-1}$ which is bounded away from zero everywhere on $S^{n-1}$) we have $\inf_{u\in S^{n-1}} h_K(u)>0$ and thus $\Lambda_g$ is bounded, hence a compact.
\end{enumerate}
\end{example}
\begin{example}
Show that  conditions (i)-(iii) of Theorem  \ref{thm:ExistSol}  are satisfied if $(X,\mathcal{X}_n)$ is a Gaussian random vector with marginal distribution $N(0,1)$. For any $\lambda\in \Lambda_g$, it holds $\widehat{X}_\lambda\sim N(0,1)$.  Condition (i) was shown in Example \ref{ex:Lambda}. Let us check condition (ii).
In view of \cite{Dasetal22}, the copula diagonal writes
$$
C _{X, \widehat{X}_\lambda}(x,x)=x^2+ \frac{1}{2\pi} \int_{0}^{\sin^{-1}\left(\rho_\lambda \right)} \exp\left(- ( \varphi^{-1}(x))^{2}\frac{1-\sin\left(\theta\right)}{\cos^{2}\left(\theta\right)}\right)d\theta,
$$
where $  \varphi^{-1}(x)$ is the quantile function of $N(0,1)$ and  $\rho_{\lambda} = \mathrm{Corr}\left(X,\widehat{X}_\lambda \right)=\sum_{j=1}^n \lambda_j \mathrm{Cov}\left(X,{X}_j\right)$. 
Since the exponential function under the integral is nonnegative and bounded from above by one, we get 
$$
\left|  C _{X, \widehat{X}_{\lambda_1}}(x,x)-C _{X, \widehat{X}_{\lambda_2}}(x,x)  \right|\le   \frac{1}{2\pi} \left|   \sin^{-1}\left(\rho_{\lambda_1}\right) - \sin^{-1}\left(\rho_{\lambda_2}\right)   \right|, \quad \lambda_!, \lambda_2 \in \Lambda_g
$$
uniformly on $x\in [0,1]$ which shows the uniform continuity of $C _{X, \widehat{X}_\lambda}(x,x)$ on $\Lambda_g$.
To show (iii), let the covariance matrix $\Sigma$ of  $\mathcal{X}_n$ be positive definite. Then $\widehat{X}_{\lambda}\sim N(0, \lambda^\top \Sigma \lambda)$ where $\lambda^\top \Sigma \lambda>0$ for all $\lambda\neq 0$, hence it has a Gaussian density which is continuous on $\R^n\setminus\{0\}$ in the $L^1$--norm. To see this, just use the multivariate mean value theorem for this density with respect to $\lambda$.

\end{example}

The next result shows that it is sufficient to consider bounded spaces $\Lambda$ or $\Lambda_g$ in minimization problems \eqref{unconst:eq1}, \eqref{const:eq2}, and \eqref{eq:FinalPreidct}. If these spaces are additionally closed, the existence of a solution is guaranteed by Theorem  \ref{thm:ExistSol}. For instance, so is  often the choice
$$
\Lambda=\{  \lambda\in \R^n: |\lambda|\le M \} \mbox{ or }  \Lambda=\{  \lambda\in \R_+^n: |\lambda|\le M \}
$$
for a suitable $M>0$. % In what follows, the weight set $\Lambda$  can be also (if needed) replaced by the constrained set $\Lambda_g$. 
Introduce the notation 
$$
\Phi_1(\lambda):=\E F_X (X\vee \widehat{X}_\lambda)  , 
$$
$$
\Phi_2(\lambda):=2\E F_X (X\vee \widehat{X}_\lambda)-\E F_X (\widehat{X}_\lambda)   , 
$$

$$
\Phi_3(\lambda):=\Phi_2(\lambda) +\gamma \left[ \E F_X^2(\widehat{X}_\lambda)- \E [F_X(\widehat{X}_\lambda)\vee Y] \right]    
$$
for the target functionals in minimization problems \eqref{const:eq2}, \eqref{unconst:eq1}, \eqref{eq:FinalPreidct} or \eqref{eq:FinalPreidct1}, respectively.

\begin{proposition} \label{thm:Bounded}
Assume that there exists $\lambda _0 \in \Lambda_g$ or $\Lambda$ such that $ \Phi_j(\lambda_0)<1$, $j=1,2,3$.
Let $\widehat{X}_{\lambda_k} \toP + \infty$ as $k\to\infty$ for any sequence $\{\lambda_k\}\subset \Lambda$  such that $|\lambda_k| \to + \infty$. Then there exists $M>0$ such that
 $$
  \min_{\lambda\in\Lambda_g  }\Phi_1(\lambda)= \min_{\lambda\in\Lambda_g: |\lambda|\le M }\Phi_1(\lambda), \quad 
 \min_{\lambda\in\Lambda  }\Phi_j(\lambda)= \min_{\lambda\in\Lambda: |\lambda|\le M }\Phi_j(\lambda), \quad j=2,3.
 $$
\end{proposition}
\begin{proof}
Sequences of random variables $\{  F_X (\widehat{X}_{\lambda_k})  \}$,  $\{  F_X^2 (\widehat{X}_{\lambda_k})  \}$, $\{  F_X (X\vee \widehat{X}_{\lambda_k})  \}$,  $\{  F_X(\widehat{X}_{\lambda_k})\vee Y  \}$
are uniformly integrable since they are a.s.  bounded by zero and one. Hence, their expectations tend to one as  $|\lambda_k| \to + \infty$ by properties of a c.d.f. Then it holds
$$
\Phi_j(\lambda_k) \to 1, \quad   |\lambda_k| \to + \infty, \quad j=1,2,3.
$$
Take $M>0$ such that $\Phi_j(\lambda_k) >  \Phi_j(\lambda_0)$ for all $k$ such that $|\lambda_k|>M$,   $j=1,2,3$.
The assertion is proven.
\end{proof}

\begin{example}  Assume that there exists a $\lambda_0 \in\Lambda_g$ or $\Lambda$ such that
\begin{enumerate}
\item  $ \widehat{X}_{\lambda_0}=X$ a.s. This is the case for some  prediction functions $g$ if $X=X_{j_0}$ a.s., $j_0\in\{ 1,\ldots,n\}$. Then it can be easily shown that
$ \Phi_1(\lambda_0)= \Phi_2(\lambda_0)=1/2$, $ \Phi_3(\lambda_0)=  1/2-\gamma/3<1$ for all $\gamma>0$. 
\item   $ \widehat{X}_{\lambda_0}$ and $X$ are stochastically independent.  This can be the case if $X,\mathcal{X}_n$ form an $m$--dependent sequence with $m<n$. Then it can be easily shown that
$ \Phi_1(\lambda_0)= 2/3$ (cf. Example \ref{ex:GiniIndependent}), 
$$\Phi_2(\lambda_0)=1+\int_0^1 F_{Y_1}(x)(1-2x)dx <1$$ with $Y_1=F_X(\widehat{X}_{\lambda_0}),$ because
$\int_0^1 F_{Y_1}(x)(1-2x)dx=\int_0^1 (F_{Y_1}(x)-F_{Y_1}(1/2))(1-2x)dx <0$, and
$$\Phi_3(\lambda_0)\leq \Phi_2(\lambda_0) <1.$$
%$$ \Phi_3(\lambda_0)=  1+\int_0^1 F_{Y_1}(x)\left(  \gamma F_{Y_1}(x) - (2\gamma+1) x  \right) dx <1$$ for $\gamma \in (0, \gamma_1)$ where
% $\gamma_1= \frac{    \int_0^1 x F_{Y_1}(x) \, dx    }{  \int_0^1 F_{Y_1}(x)\left( 2 x - F_{Y_1}(x) \right) dx }$. 
\end{enumerate}
%Notice that in case of the minimization problem \eqref{eq:FinalPreidct} or \eqref{eq:FinalPreidct1} the choice of constant $\gamma>0$ matters. Only relatively small values of $\gamma$
%are allowed in order to guarantee the existence of a solution. 
\end{example}
\begin{example}  
Condition $\widehat{X}_{\lambda_k} \toP + \infty$ as $k\to\infty$ for any sequence $\{\lambda_k\}\subset \Lambda$  such that $\lambda_k=( \lambda_k(1),\ldots,\lambda_k(n))$,  $|\lambda_k| \to + \infty$ is satisfied for
$\Lambda=\R^n_+$, $\widehat{X}_{\lambda_k} =\sum_{j=1}^n \lambda_k(j)X_j$ or $\widehat{X}_{\lambda_k}=\max_{j=1,\ldots, n} \lambda_k(j)X_j$ and a.s. nonnegative random variables  $X_j$, $j=1,\ldots,n$.

\end{example}

The question of  uniqueness of a solution $\lambda$ to problems  \eqref{unconst:eq1},  \eqref{const:eq2}  and \eqref{eq:FinalPreidct} cannot be resolved in such generality. As illustrated  in the Gaussian case \cite{Dasetal22}, it will require further specification of the dependence structure of observations $X_j$ within the set $\mathcal{X}_n$, of the statistic $g$ and parameter set $\Lambda$. This will be done for some specific classes of random variables in forthcoming research.

%-----------------------------------------------------------------------------------------------------------------------------------------------------------------------------------------------------------------------

\section{Excursion-based extrapolation of stationary heavy--tailed random fields}\label{sectExtrapol}

In this section, we will use the above  prediction approach  to extrapolate a real--valued strictly stationary ergodic random field  $X=\{X(t),t\in \R^d\}$ with absolutely continuous (but  possibly heavy--tailed) marginal distribution $F_{\theta_0}.$ We assume that $F_{\theta_0}$ belongs to an appropriate parametric family of possible marginal distributions $\{ F_\theta, \theta\in\Theta\}$, $\Theta\subseteq \R^k$. This ansatz can be useful, in particular, for heavy--tailed time series forecasting in insurance/finance ($d=1$, compare Section \ref{sect:Num})   or in image analysis ($d=2,3$) for upscaling of low-resolution 2D and 3D gray scale images (cf.  the so-called {\it super-resolution problem} \cite{Chang04,KimKwon10,Freeman11c.:markov,Dong16}).

Denote by $\Z_h=(h_1 \Z)\times \cdots \times (h_d \Z) $ the $d-$dimensional grid with mesh sizes $h=(h_1,\ldots,h_d)\in (0,+\infty)^d.$
Let $X$ be potentially observed at points $\mathbb{T}_0:=W_o\cap \Z_h,$ where $W_o\subset \R^d$ is a compact.  The observed values form a sample
 $\mathcal{X}_T:=\{X(t_j),t_j\in \mathbb{T}_0\}.$
  
  Let us predict the value $X(t)$ at a location $t \in \Z_h$, $t \not \in W_o$ from the knowledge of the so--called {\it forecast sample} $T_f:= \{t_1,\ldots,t_n\} \neq \emptyset,$ $T_f\subset \Z_h.$ The predictor $\widehat{X}_\lambda=g(\lambda,\mathbf{X}(T_f))$ with $\lambda\in \Lambda\subseteq\R^n$ and $\mathbf{X}(T_f):=(X(t_1),\ldots,X(t_n))^{\top}$ requires weights $\lambda$  to be a solution of minimization problems  \eqref{unconst:eq1},
 \eqref{eq:FinalPreidct} or \eqref{eq:FinalPreidct1}, i.e.,
 \begin{equation}\label{eq:minFctl_noConstraint}
\hat{\lambda}=\arg \min_{\lambda\in \Lambda}\left[2\E  \left[ F_{\theta_0} (X(t))\vee F_{\theta_0} ( \widehat{X}_\lambda) \right] - \E F_{\theta_0}(\widehat{X}_\lambda)  \right] ,
\end{equation}
 \begin{equation}\label{eq:minFctl}
\hat{\lambda}=\arg \min_{\lambda\in\Lambda}\left\{ 2\E  \left[ F_{\theta_0} (X(t))\vee F_{\theta_0} ( \widehat{X}_\lambda) \right]  - \E F_{\theta_0}(\widehat{X}_\lambda)  +\gamma 
\left[\E F_{\theta_0}^2(\widehat{X}_\lambda) -  \E [F_{\theta_0}(\widehat{X}_\lambda)\vee Y]  \right]   \right\} ,   \mbox{ or}
\end{equation}
 \begin{equation}\label{eq:minFctl1}
\hat{\lambda}=\arg \min_{\lambda\in\Lambda}\left\{ 2 \E  \left[   F_{\theta_0} (X(t))\vee F_{\theta_0} ( \widehat{X}_\lambda) \right]   - \E F_{\theta_0}(\widehat{X}_\lambda)  + \gamma 
\left[  \int_0^1 F_{F_{\theta_0} ( \widehat{X}_\lambda)}(y) \left[ F_{F_{\theta_0} ( \widehat{X}_\lambda)}(y)-2y\right]\, dy  \right]   \right\} ,   
\end{equation}
respectively, due to strict monotonicity of  $F_{\theta_0}$, where $Y$ is an independent copy of $F_{\theta_0}(\widehat{X}_\lambda)$ and $F_{F_{\theta_0} ( \widehat{X}_\lambda)}$ is the c.d.f. of the random variable $F_{\theta_0} ( \widehat{X}_\lambda)$.
As already mentioned in Section \ref{sectPredict}, we use  $$\widehat{X}_\lambda=\lambda^{\top}\mathbf{X}(T_f)$$ for infinitely divisible $X$ and    
  $$\widehat{X}_\lambda=\max_{j=1,\ldots, n}\lambda(j) X(t_j)$$ with $\lambda=(\lambda(1),\ldots,\lambda(n))^\top$ for max--stable $X$.

The excursion predictors \eqref{const:eq2}, \eqref{eq:minFctl_noConstraint} -- \eqref{eq:minFctl1}  are consistent under very mild assumptions. 
\begin{theorem}
Let the random field $X=\{X(t),t\in\R^d\}$  be stochastically continuous. Assume that there exists $\tilde{\lambda}_k\in \Lambda_g$ such that $g(\tilde{\lambda}_k,\mathbf{X}(T_f))=X(t_k)$ a.s.  for any $k=1,\ldots,n$  and
$\min_{j=1,\ldots,n} \|t_j - t\|_2 \to 0$ as $n\to \infty.$ Then
$\widehat{X}_{\hat \lambda}(t)\stackrel{P}{\to}X(t)$ as $n\to \infty,$ where
$\widehat{X}_{\hat \lambda}$ is an excursion predictor  \eqref{const:eq2}, \eqref{eq:minFctl_noConstraint},  \eqref{eq:minFctl} or \eqref{eq:minFctl1}.
\end{theorem}
\begin{proof}
For the method \eqref{eq:minFctl_noConstraint}, let
$\widehat{X}_{\hat \lambda}=g(\hat{\lambda},\mathbf{X}(T_f))$,  where $
    \hat{\lambda}:=\arg \min_{\lambda\in \Lambda}\left[2\E F_X (X\vee \widehat{X}_\lambda)-\E F_X (\widehat{X}_\lambda)\right].$
Denote by $\tilde{\lambda}_n\in \Lambda_g\subset \Lambda$ such that $g(\tilde{\lambda}_n,\mathbf{X}(T_f))=X({\tilde{t}_{n}}),$ where  $\tilde{t}_{n}=\arg\min_{j=1,\ldots,n} \|t_j - t\|_2.$  Then
excursion metric writes
\begin{align*}
    E_{F_X}(\widehat{X}_{\hat \lambda}(t),X(t))&=\E |F_X (X(t))-F_X( \widehat{X}_{\hat \lambda}(t))|
    =\min_{\lambda\in \Lambda}\left[ 2  \E F_X (X(t)\vee  \widehat{X}_{ \lambda}(t) )-\E F_X ( \widehat{X}_{ \lambda}(t) )\right]-\frac{1}{2}\\
    &\leq 2 \E F_X (X(t)\vee  \widehat{X}_{ \tilde{\lambda}_n}(t))-\E F_X (\widehat{X}_{ \tilde{\lambda}_n}(t))-\frac{1}{2}
    = 2\E F_X (X(t)\vee X(\tilde{t}_n))-\E F_X (X(\tilde{t}_n))-\frac{1}{2}\\
    &=\E |F_X (X(t))-F_X( X(\tilde{t}_n))|.
\end{align*}
Sequence $\{F_X (X(t))-F_X( X(\tilde{t}_n))\}_{n\geq1}$ is obviously uniformly integrable and  $F_X( X(\tilde{t}_n))\stackrel{P}{\to}F_X(X(t))$ as $\tilde{t}_n\to t.$ Therefore, $F_X( \widehat{X}_{\hat \lambda}(t)) \to F_X(X(t))$ in $L^1$-- sense and, consequently, $F_X(\widehat{X}_{\hat \lambda}(t))\stackrel{P}{\to}F_X(X(t)).$ Due to the continuous mapping theorem, $\widehat{X}_{\hat \lambda}(t)\stackrel{P}{\to}X(t)$ as $n\to \infty.$

The excursion predictor given by \eqref{const:eq2}  is consistent as well by similar arguments: 
 \begin{align*}
    E_{F_X}(\widehat{X}_{\hat \lambda}(t),X(t))&=\E |F_X (X(t))-F_X( \widehat{X}_{\hat \lambda}(t))|
    =\min_{\lambda\in \Lambda_g}\left[ 2  \E F_X (X(t)\vee  \widehat{X}_{ \lambda}(t) )-1\right] \\
    &\leq 2 \E F_X (X(t)\vee  \widehat{X}_{ \tilde{\lambda}_n}(t))-1
    = 2\E F_X (X(t)\vee X(\tilde{t}_n))-1
    =\E |F_X (X(t))-F_X( X(\tilde{t}_n))|\to 0
\end{align*}
as $n\to \infty$, where $\tilde{\lambda}_n\in \Lambda_g$ such that $g(\tilde{\lambda}_n,\mathbf{X}(T_f))=X(\tilde{t}_n)$ a.s.

As for the excursion predictor  \eqref{eq:minFctl}, choose a  $\gamma>0$ and write for $Y_1=F_X ( \widehat{X}_{ \hat\lambda}(t))$, $Y_2\sim U(0,1)$
 that 
 \begin{align*}
    E_{F_X}(\widehat{X}_{\hat \lambda}(t),X(t))+ \gamma \rho\left( F_{Y_1}, F_{Y_2} \right)    &=     2 \E F_X (X(t)\vee  \widehat{X}_{ \hat{\lambda}}(t))-\E F_X (\widehat{X}_{ \hat{\lambda}}(t)) \\
    &+ \gamma\left[     \E F_X^2 ( \widehat{X}_{\hat \lambda}(t) )    -\E (F_X (  \widehat{X}_{ \hat \lambda}(t) )  \vee Y ) \right] +\frac{\gamma}{3}   -\frac{1}{2}\\
    &=\min_{\lambda\in \Lambda}\left[ 2  \E F_X (X(t)\vee  \widehat{X}_{ \lambda}(t) )-\E F_X ( \widehat{X}_{ \lambda}(t) ) 
   \right. \\
    &\left. + \gamma\left(     \E F_X^2 ( \widehat{X}_{ \lambda}(t) )  -\E (F_X (  \widehat{X}_{  \lambda}(t) )  \vee Y ) \right) \right]  +\frac{\gamma}{3}   -\frac{1}{2}\\
    &\leq 2 \E F_X (X(t)\vee  \widehat{X}_{ \tilde{\lambda}_n}(t))-\E F_X (\widehat{X}_{ \tilde{\lambda}_n}(t))-\frac{1}{2}\\
    &+\gamma\left( \frac{1}{3}+  \E F_X^2 ( \widehat{X}_{ \tilde \lambda_n}(t) ) -\E (F_X (  \widehat{X}_{  \tilde \lambda_n}(t) )  \vee Y )  )  \right)\\
    &= 2\E F_X (X(t)\vee X(\tilde{t}_n))-\E F_X (X(\tilde{t}_n))-\frac{1}{2}\\
    & +\gamma\left( \frac{1}{3}+  \E F_X^2 ( X(\tilde{t}_n) ) -\E (F_X ( X(\tilde{t}_n))  \vee Y )  )  \right)\\
    &=\E |F_X (X(t))-F_X( X(\tilde{t}_n))|\to 0
\end{align*}
as $n\to \infty,$ since $X(\tilde{t}_n)\eqd X(t)$.
\end{proof}

If the parameter $\theta_0$ is unknown, we assess it by a statistic $\widehat \theta$  in order to find $F_{\widehat \theta},$ which is a plug-in estimator of $F_{\theta_0}.$ 
By ergodicity of $X$, we substitute expectations in \eqref{eq:minFctl_noConstraint}-\eqref{eq:minFctl1} by the corresponding empirical moments.
The  prediction problems above get the form
\begin{equation}\label{eq:EmpMoment}
     \bar\Phi_k(\lambda):=\sum_{j=1}^N   Q_j^{(k)}(\lambda)  \to \min_{\lambda\in\Lambda}, \quad k=2,3,4,
\end{equation}
where 
\begin{equation}\label{eq:Q_unconstraint}
Q^{(2)}_j(\lambda):=2F_{\widehat \theta} (X(t+h_j))\vee   F_{\widehat \theta} ( g( \lambda, \mathbf{X}(T_f+h_j)))- F_{\widehat \theta} ( g( \lambda, \mathbf{X}(T_f+h_j)))
\end{equation} 
for unconstrained prediction \eqref{eq:minFctl_noConstraint},
\begin{equation}\label{eq:Q}
Q^{(3)}_j(\lambda):=Q^{(2)}_j(\lambda)+\gamma\left[ F^2_{\widehat \theta} ( g( \lambda, \mathbf{X}(T_f+h_j)))-F_{\widehat \theta} ( g( \lambda, \mathbf{X}(T_f+h_j))) \vee Y_{j}\right]
\end{equation} 
for the  (approximatively) law-preserving prediction \eqref{eq:minFctl}, and
\begin{align}
Q^{(4)}_j(\lambda):&=Q^{(2)}_j(\lambda) + \gamma F^2_{\widehat \theta} ( g( \lambda, \mathbf{X}(T_f+h_j)))  \nonumber  \\
&-\frac{\gamma}{N} \left[  F_{\widehat \theta} ( g( \lambda, \mathbf{X}(T_f+h_j)))  +2 \sum_{i=1}^{j-1} 
F_{\widehat \theta} ( g( \lambda, \mathbf{X}(T_f+h_i))) \vee   F_{\widehat \theta} ( g( \lambda, \mathbf{X}(T_f+h_j))) \label{eq:Q1}
 \right]
\end{align} 
for the law-preserving prediction variant \eqref{eq:minFctl1}, where the convention $\sum_{i=1}^0=0$ is used. Here
$T_f+h_j$, $j=1,\ldots, N$ with  $\{h_1,\ldots,h_N\}:=\{s\in \Z_h: s+T_f\cup \{t\}\subset  \mathbb{T}_0   \}$ are the so--called {\it learning samples}, and 
 $Y_j$ are independent copies of $F_{\widehat \theta} ( g( \lambda, \mathbf{X}(T_f+h_j)))$. In practice, the sample $\{Y_1,\ldots,Y_N\}$ can be obtained for each $\lambda$ by bootstrap, i.e. resampling of 
\begin{equation}\label{eq:Bootstr}\{ F_{\widehat \theta} ( g( \lambda, \mathbf{X}(T_f+h_1))),\ldots, F_{\widehat \theta} ( g( \lambda, \mathbf{X}(T_f+h_N)))\}.\end{equation} 
 
\begin{remark}
The extrapolation methods \eqref{eq:minFctl_noConstraint}-\eqref{eq:minFctl1}  with $\theta_0$ replaced by $\theta_t$ can be also used  for the extrapolation of non--stationary random fields $X$ if their marginal distributions  $F_{\theta_t}(x)=\pr(X(t)\le x)$ are known in advance. In this case,  the empirical moments  \eqref{eq:Q_unconstraint}-\eqref{eq:Bootstr} in the problem \eqref{eq:EmpMoment} have to be rewritten with 
$F_{\theta_t}$ in lieu of $F_{\widehat \theta}$.
\end{remark}

In order to find the minimum of $\bar\Phi_k(\lambda)$, we use a subgradient descent. Assume that $F_\theta$ has a density $p_\theta.$ 
Since the marginal distribution of $X$ is absolutely continuous and $X(t+h_j)\not\in   \mathbf{X}(T_f+h_j)$ for all $j$ a.s. if the joint probability density of $(X(t+h_j),  \mathbf{X}(T_f+h_j))$ exists, it is reasonable to assume that 
 $$
 \pr\left( X(t+h_j)= g( \lambda, \mathbf{X}(T_f+h_j)) \right)=0,
 $$
 $$
 \pr\left( Y_j=  F_{\widehat \theta} ( g( \lambda, \mathbf{X}(T_f+h_j))) \right)=0, \quad j=1,\ldots, N,
 $$
$$
 \pr\left( F_{\widehat \theta} ( g( \lambda, \mathbf{X}(T_f+h_i))=  F_{\widehat \theta} ( g( \lambda, \mathbf{X}(T_f+h_j))) \right)=0, \quad i, j=1,\ldots, N.
 $$

Then the subgradients $\nabla^*Q^{(k)}_j(\lambda)$, $k=2,4$, write with probability one as
  \begin{align*}
    \nabla^*Q^{(2)}_j(\lambda)&=  \Bigg[2\mathbbm{1}\{X(t+h_j)<g( \lambda, \mathbf{X}(T_f+h_j))  \} -1\Bigg]   p_{\widehat \theta}(g( \lambda, \mathbf{X}(T_f+h_j)) )  \nabla^* g( \lambda, \mathbf{X}(T_f+h_j)),
\end{align*}
 \begin{align*}
    \nabla^*Q^{(4)}_j(\lambda)&=  \nabla^*Q^{(2)}_j(\lambda) +\gamma  \Bigg[ 2 F_{\widehat \theta}( g( \lambda, \mathbf{X}(T_f+h_j)))  - \frac{1}{N}  -   \frac{2}{N} \sum_{i=1}^{j-1} \mathbbm{1}\{F_{\widehat \theta} ( g( \lambda, \mathbf{X}(T_f+h_i))< F_{\widehat \theta} ( g( \lambda, \mathbf{X}(T_f+h_j)))\}    \Bigg. \\
    & -  \left. \frac{2}{N} \sum_{i=1}^{j-1} \mathbbm{1}\{F_{\widehat \theta} ( g( \lambda, \mathbf{X}(T_f+h_i))> F_{\widehat \theta} ( g( \lambda, \mathbf{X}(T_f+h_j)))\} \right] p_{\widehat \theta}(g( \lambda, \mathbf{X}(T_f+h_i)) )  \nabla^* g( \lambda, \mathbf{X}(T_f+h_i)),
\end{align*}
respectively. The subgradient $\nabla^*Q^{(3)}_j(\lambda)$  can be written as 
\begin{align*}
    \nabla^*Q^{(3)}_j(\lambda)&= \nabla^*Q^{(2)}_j(\lambda) +\gamma \Bigg[ 2 F_{\widehat \theta}( g( \lambda, \mathbf{X}(T_f+h_j)))  ) - \mathbbm{1}\{g( \lambda, \mathbf{Y}(T_f+h_j))<  g( \lambda, \mathbf{X}(T_f+h_j))\}   \Bigg] \\
    &
    \times  p_{\widehat \theta}(g( \lambda, \mathbf{X}(T_f+h_j)) )  \nabla^* g( \lambda, \mathbf{X}(T_f+h_j))\\
    &
    -\gamma  \mathbbm{1}\{g( \lambda, \mathbf{Y}(T_f+h_j))\geq  g( \lambda, \mathbf{X}(T_f+h_j))\}   p_{\widehat \theta}(g( \lambda, \mathbf{Y}(T_f+h_j)) )  \nabla^* g( \lambda, \mathbf{Y}(T_f+h_j)),
\end{align*}
where $\mathbf{Y}$ is an independent copy of $\mathbf{X}.$

% Old version:

% the subgradient for $Q_j$  is 
%\begin{align*}
   % \nabla^*Q_j(\lambda)&=\Bigg[ \mathbbm{1}\{X(t+h_j)= g( \lambda, \mathbf{X}(T_f+h_j))\}  
   % +2\mathbbm{1}\{X(t+h_j)<g( \lambda, \mathbf{X}(T_f+h_j))  \} -1 \Bigg] \\
     %   & \times p_{\widehat \theta}(g( \lambda, \mathbf{X}(T_f+h_j)) ) \nabla^* g( \lambda, \mathbf{X}(T_f+h_j))
%\end{align*}
%in case of \eqref{eq:Q_unconstraint},
%\begin{align*}
  %  \nabla^*Q_j(\lambda)&=\Bigg[\frac{1}{2} \mathbbm{1}\{X(t+h_j)= g( \lambda, \mathbf{X}(T_f+h_j))\}  \Bigg. 
  %  +\mathbbm{1}\{X(t+h_j)<g( \lambda, \mathbf{X}(T_f+h_j))  \} \\
  %  &+\frac{\gamma}{2}
   % \mathbbm{1}\{Y_j=  F_{\widehat \theta} ( g( \lambda, \mathbf{X}(T_f+h_j)))\}  
 %   +\gamma \mathbbm{1}\{Y_j< F_{\widehat \theta} ( g( \lambda, \mathbf{X}(T_f+h_j)))\} \\
  %  &
   % -\Bigg.2\gamma F_{\widehat \theta}( g( \lambda, \mathbf{X}(T_f+h_j)))\Bigg] p_{\widehat \theta}(g( \lambda, \mathbf{X}(T_f+h_j)) ) \nabla^* g( \lambda, \mathbf{X}(T_f+h_j))
%\end{align*}
%for \eqref{eq:Q},  and

%\begin{align*}
%    \nabla^*Q_j(\lambda)&=\Bigg[\frac{1}{2} \mathbbm{1}\{X(t+h_j)= g( \lambda, \mathbf{X}(T_f+h_j))\}  \Bigg. 
   % +\mathbbm{1}\{X(t+h_j)<g( \lambda, \mathbf{X}(T_f+h_j))  \} \\
  %  &+\frac{\gamma}{2}
   % \mathbbm{1}\{Y_j=  F_{\widehat \theta} ( g( \lambda, \mathbf{X}(T_f+h_j)))\}  
   % +\gamma \mathbbm{1}\{Y_j< F_{\widehat \theta} ( g( \lambda, \mathbf{X}(T_f+h_j)))\} \\
  %  &
   % -\Bigg.2\gamma F_{\widehat \theta}( g( \lambda, \mathbf{X}(T_f+h_j)))\Bigg] p_{\widehat \theta}(g( \lambda, \mathbf{X}(T_f+h_j)) ) \nabla^* g( \lambda, \mathbf{X}(T_f+h_j))
%\end{align*}
% for \eqref{eq:Q1}, respectively. 

Here  is  e.g.  $    \nabla^* g( \lambda, \mathbf{X}(T_f+h_j))=   \mathbf{X}(T_f+h_j)$ for infinitely divisible $X$ and 
 $$    \nabla^* g( \lambda, \mathbf{X}(T_f+h_j))=  \left(   \mathbf{X}(t_i+h_j) \mathbbm{1}\left\{ \lambda^{(i)} X(t_i+h_j)=  \max_{k=1,\ldots, n}\lambda^{(k)} X(t_k+h_j)  \right\} , \quad i=1,\ldots, n \right)$$
 for max--stable $X$.
 
 %Assume that $g(\lambda,\cdot)$ is a.s.
 %Lipschitz in $\lambda$ with the Lipschitz constant being integrable w.r.t. the distribution of $\mathbf{X}(T_f+h_j)$. Then the functional $Q_j(\lambda)$ satisfies the condition \cite[(3.8)]{Bottou98} stating that there exist an integrable random function $G_\lambda$ and a neighborhood $B_o$ of zero in $\R^n$ such that 
 %$$
 %| Q_j(\lambda + \Delta \lambda) - Q_j(\lambda) |\le |\Delta \lambda| G_\lambda, \quad \Delta \lambda \in B_o
 %$$
 % a.s. w.r.t. the joint distribution  of $(X(t+h_j),  \mathbf{X}(T_f+h_j))$.
 
 \begin{remark}
 The speed of convergence of $ \frac{1}{N} \sum_{j=1}^N   Q_j^{(k)}(\lambda) $ to  its expectation $   \Phi_k(\lambda)  $ as $N\to\infty$ in the ergodic theorem for correlated data $Q_j^{(k)}$ highly depends on their correlation rate. Hence, large values of $N$ ($N\approx 1000$) are recommended for practical use.
  {In the case of random processes with infinite variance, the speed of convergence can be determined via their $\beta-$mixing properties, see e.g \cite{yu1994rates}. }
 \end{remark}
 
Now use the classical (batch) subgradient descent \cite{Kiwiel01} with  e.g. $\lambda_0\in\Lambda$,
\begin{equation}\label{eq:BatchGr}
\lambda_{l+1}=\Pi_\Lambda \left[ \lambda_l-\frac{\eta_l}{N} \sum_{j=1}^N  \nabla^*Q_j(\lambda_l)\right] , \quad l\in \N,\end{equation} 
where $\Pi_\Lambda[\cdot]$ is the metric projection onto $\Lambda$ and $\eta_l>0$ is a step length factor which has to be tuned numerically.   The iterations stop at some $l^*$ whenever $| \lambda_{l^*+1}-\lambda_{l^*} |<\delta$ for some small threshold value $\delta>0$ yielding $\widehat \lambda=\lambda_{l^*}$.
In order to avoid costly computations at each step, a stochastic (or online) subgradient descent \cite{Bottou98} may be performed instead. Here, at each step
$ l\in \N,$ we do
\begin{equation}\label{eq:OnlineGr}
\lambda_{l+1}=\Pi_\Lambda \left[ \lambda_l-{\eta_l}   \nabla^*Q_j(\lambda_l)\right] , \end{equation} 
where $j$ is chosen at random uniformly from  $\{1,\ldots, N\}$.
For the sequence $ \{  \eta_l\}$, we may require  
$$ \sum\limits_{l=1}^\infty \eta_l=\infty, \quad \sum\limits_{l=1}^\infty \eta_l^2<\infty, 
$$ 
for instance, $\eta_l=l^{-1}$. 
 In addition, the {\it Polyak-Ruppert averaging} can be used after a burn--in period of length $l_0$: the resulting weight vector 
$$
\widehat \lambda= \frac{1}{l-l_0} \sum\limits_{l=l_0}^{l^*-1} \lambda_l.
$$
Alternatively, we may set $\widehat \lambda$ to be equal to  the value of $ \lambda_l$ with the smallest target functional $\Phi(\lambda_l)$.
\begin{remark} \label{Rem:6} In general, we assume  the weight space $\Lambda$ to be a convex cone within $\R^n$. A more accurate choice of $\Lambda$ should reflect the constraints onto the support of the distribution of $X(t)$.   For instance, if $X(t)\ge 0$ a.s. we may take $\Lambda=\R_+^n$. However, for practical reasons of avoiding back projection $\Pi_\Lambda$ onto $\Lambda$ at each iteration step \eqref{eq:BatchGr} or \eqref{eq:OnlineGr} of subgradient descent methods,  it is better to modify the predictor $\widehat X_\lambda$ and make $g( \lambda, \cdot)$ be dependent on $\lambda^2=(\lambda^2(1),\ldots,\lambda^2(n))^\top$  instead of $\lambda=(\lambda(1),\ldots,\lambda(n))^\top$. Doing so, the formulas for the subgradient  $ \nabla^* g( \lambda, \mathbf{X}(T_f+h_j))$ have to be modified accordingly.
\end{remark}

\begin{remark} \label{Rem:7}  The advantage of  optimization formulation   \eqref{eq:Q1} in comparison with \eqref{eq:Q} is that a bootstrap step (generating instances $Y_j$ and thus increasing the variance of the forecast) is not needed for the evaluation of the subgradient $ \nabla^*Q_j(\lambda)$. However, there is a fee to pay: a more slow calculation of $ \nabla^*Q_j(\lambda)$ due to the sum inside.
\end{remark}

It is worth mentioning that many existing optimization routines (e.g. those built in R or Mathlab) can be used to minimize the functional \eqref{eq:EmpMoment} in lieu of \eqref{eq:BatchGr} or \eqref{eq:OnlineGr}. They sometimes work more accurately but are rather slow, cf. Table \ref{T:RT}.
Under several additional assumptions, the  a.s. convergence of the stochastic gradient descent method \eqref{eq:OnlineGr}  can be shown; however, this would blow up the length of this paper and thus will be the matter of future papers.

\section{Numerical examples}\label{sect:Num}

In this section, we test our prediction methods on  simulated data.
Although our approach works for random fields on $W\subset \R^d$, we take $d=1$ in order to simplify computations and the representation of results.

A random process $X$ is observed at points $\mathbb{T}_0=W_{o}\cap \mathbb{Z}_{h_1} \cup T_f,$ where $h_1=0.02,$ $W_{o}=[0,30-h_1],$ and $T_f$ is the forecast sample. We take $T_f=\{30.0,30.1,\ldots,30.9\}$ for the extrapolation and $T_f=\{30.0,30.5,31.0\ldots,34.5\}$ for the interpolation problems. In both cases $n=|T_f|=10.$

We predict the values $X(t)$ at locations $t\in \mathbb{Z}_{h_1}\cap[30,35],$ $t\not \in \mathbb{T}_0$ via predictor $\hat{X}_\lambda=\lambda^{\top}\mathbf{X},$ where $\lambda\in \R^n$ and $\mathbf{X}=(X(t_1),\ldots,X(t_n))^{\top},$ $t_j\in T_f,$ $j=1,\ldots,n.$

We solve the arising minimization problems by the   stochastic subgradient descent method from Section~\ref{sect:Num}. Our preliminary numerical studies show that {$\gamma=5$ is a good value for the constrained optimization.} The minimization sequence $\Phi(\lambda_l)$ obtained by  the classical (batch) subgradient descent very often stacks in some local minima. A stochastic (or online) subgradient descent has much better performance in a sense that $\bar\Phi_k(\lambda_l)$ reaches lower levels. After a series of numerical experiments, we can recommend the use of the sequence $\eta_l=10 (10+l)^{-\beta}$ with $\beta=0.7$ in \eqref{eq:OnlineGr} {with $l\leq 300$}. Moreover, the value of $\beta$ has a two-sided effect. Decreasing  $\beta$, the volatility of $\lambda_l$ increases, which produces more possibilities of gaining  a global minimum. But then the sequence of $\lambda_l$ converges slower, and the number of computational steps increases as well.

One can also use the result of the stochastic subgradient descent method as an initial value for other optimization routines. This combines the advantages of two procedures, but increases the runtimes. 

\begin{table}[]
    \centering
    \begin{tabular}{cccc}
         Method & Stochastic subgradient (300 iterations) & R: optim & Wolfram Mathematica: NMinimize\\
         \hline
         Runtime & 0.9 & 7.1 & 254\\
         \hline
    \end{tabular}
    \caption{The runtime for the solution of minimization problem  \eqref{eq:minFctl_noConstraint} at one point $t.$  The CPU times are given in seconds for a PC with an Intel(R) Core(TM) i5,  CPU 3.5 GHz Quad-Core processor and 32 GB RAM.}
    \label{T:RT}
\end{table}

The  choice of an initial value $\lambda_0$  is crucial for the good convergence of \eqref{eq:OnlineGr}. 
 Based on our experience, we provide the following practical recommendations. First, produce a finite number of "candidates" $\lambda_{0,j}$  for $\lambda_0.$ 
 Then the initial value $\lambda_0$ is chosen as $\mbox{argmin}_{j} \bar\Phi_k(\lambda_{0,j})$.
One possible set of such candidates may be  $\lambda_{0,j}=(0,\ldots,1,\ldots,0),$ $j=1,\ldots, n$ {or the value obtained from the optimization problem for the neighbour point $t$.}
Another  one may  consist of a fixed number  of $\lambda$'s generated randomly on $[0,1]^n$ such that $\|\lambda\|_1=1.$

We choose three models of stationary infinitely divisible random processes: Gaussian stochastic process, moving average with $\alpha$-stable marginals, and an autoregressive model with Student $t-$distributed innovations.

As Gaussian random processes are well studied and their behaviour is determined by the covariance function, they allow us to compare the performance of our prediction method via excursions and some popular procedures, cf. kriging.

In the $\alpha$--stable moving average case, we model the dependence within  ${\bf X}$ by a deterministic kernel function and determine the marginal distribution of $X$ via the choice of a random integrator measure.  If $\alpha\in(0,2)$,  the variance of $X(t)$ is infinite, and the $L^2$--forecasting techniques are not applicable.

We simulate also an autoregressive model in order to study two effects: the accuracy of the solutions of the minimization problems, and the method's performance without knowing the  marginal distribution.

The R code for our prediction methods can be found in \cite{RCode22Makogin}.

{The marginal distributions functions are taken as $F_{\hat{\theta}}$ from the corresponding parametric family, and parameters' estimates $\hat{\theta}$ are obtained from the one sample trajectories $X(t),t\in W_o \cap \Z_{h_1}.$}

{We also do not solve the minimization problems for the times points from the forecast sample $T_f.$ Naturally, we put $\hat{X}(t_k)=X(t_k),$ $t_k\in T_f$. We see from the further plots that the predicted trajectories are continuous functions and $\hat{X}(t)\approx X(t_k)$ if $t$ is close to $t_k.$ Therefore, we can avoid computations for $t\in T_f.$}

%--------------------------------------------------------------------------------------------------------
\subsection{Gaussian random processes}
In the case of a Gaussian random process $X,$ the exact solution $\hat{\lambda}_e(t)$ of minimization problem \eqref{const:eq2}  is given in \cite[Theorem 3.5]{Dasetal22} by
$$\hat{\lambda}_e(t)=\sqrt{\Var X(t)}\frac{\Sigma^{-1}c_t}{\sqrt{c_t^{\top}\Sigma^{-1}c_t}},$$
where $\Sigma$ is the covariance matrix of $\mathbf{X}$ and $c_t=\left( \Cov(X(t), X(t_1)),\ldots,  \Cov(X(t), X(t_n))   \right) .$
For numerical illustration, we take $X$ with standard normal marginal distribution and covariance function $C(t)=e^{-|t|/2},\, t\in\R.$

For each $T_f,$  we find numerical solutions $\hat{\lambda}_u(t)$ of \eqref{eq:minFctl_noConstraint} and $\hat{\lambda}_c(t)$ of \eqref{eq:minFctl},  $t\in \mathbb{Z}_{h_1}\cap[30,35] \setminus T_f$. We compare the corresponding predicted trajectories $\hat{X}_u$ and $\hat{X}_c$ with $\hat{X}_e$ obtained via \cite[Theorem 3.5]{Dasetal22}  and simple kriging $\hat{X}_{sk},$  see Figures \ref{T:Gauss:Fig:int} (interpolation) and \ref{T:Gauss:Fig:extr} (extrapolation).

\begin{remark}
{The extrapolated trajectory $\hat{X}_e$ becomes constant shortly after  the last point of observation $t_n$. Mathematically, $\hat{\lambda}_e(t)\to (0,\ldots,0,1)^\top$ as $t\gg t_n.$ Indeed, $\Sigma$ does not depend on $t$ and $\frac{c_t}{C(t-t_n)}=\left(e^{-(t_n-t_1)/2},\ldots,1\right)^\top=c_{t_n}$ for $C(t)=e^{-|t|/2}.$
Thus, $$\Sigma\hat{\lambda}_e(t)=\sqrt{\Var X(t)} \frac{c_t}{C(t-t_n)}\left(\frac{c_t^\top}{C(t-t_n)}\Sigma^{-1}\frac{c_t}{C(t-t_n)}\right)^{-1/2}= \sqrt{\Var X(0)}c_{t_n} \left(c_{t_n}^\top\Sigma^{-1}c_{t_n}\right)^{-1/2},$$ 
and, consequently, $\hat{\lambda}_e(t)= (0,\ldots,0,1)^\top.$}
\end{remark}

One can observe that the trajectories of $\hat{X}_u$ and $\hat{X}_{sk}$ are relatively close, which may indicate that the solution of unconstrained minimization problem \eqref{eq:minFctl_noConstraint} approximates the minimizers obtained by the simple kriging method.

Figure \ref{T:Gauss:Fig:int} also shows that the trajectory $\hat{X}_c$ is not so close to $\hat{X}_e.$ This effect has two sources. First, there is no exact constraint of the equality of marginal distributions of the predictor and the random process. Second, the minimization functional is approximated by its sample mean. Hence, one should increase the size $N$ of the learning sample and the weight $\gamma$ in order to obtain a closer match between $\hat{\lambda}_c(t)$ and $\hat{\lambda}_e(t).$

While the prediction weights $\hat{\lambda}_u$ and $\hat{\lambda}_c$ are computed based on one learning sample, the quality of prediction is evaluated on 1000 independently simulated trajectories of $X$ on $[30,35]\cap \Z_{h_1}.$ We compute  the corresponding sample values of excursion metric $E_{F_X}(X(t),\hat{X}(t))$ presented in Figures \ref{G:Gauss:Fig:int} and \ref{G:Gauss:Fig:extr}. The similarity between the marginal distributions  of $X(t)$ and $\hat{X}(t)$  is measured by Wasserstein distance $\rho(F(X(t),F(\hat{X}(t))),$ whose values can be found in Figures \ref{W:Gauss:Fig:int} and \ref{W:Gauss:Fig:extr}.

\begin{figure}[p]
    \centering
    \includegraphics[width=0.8\linewidth]{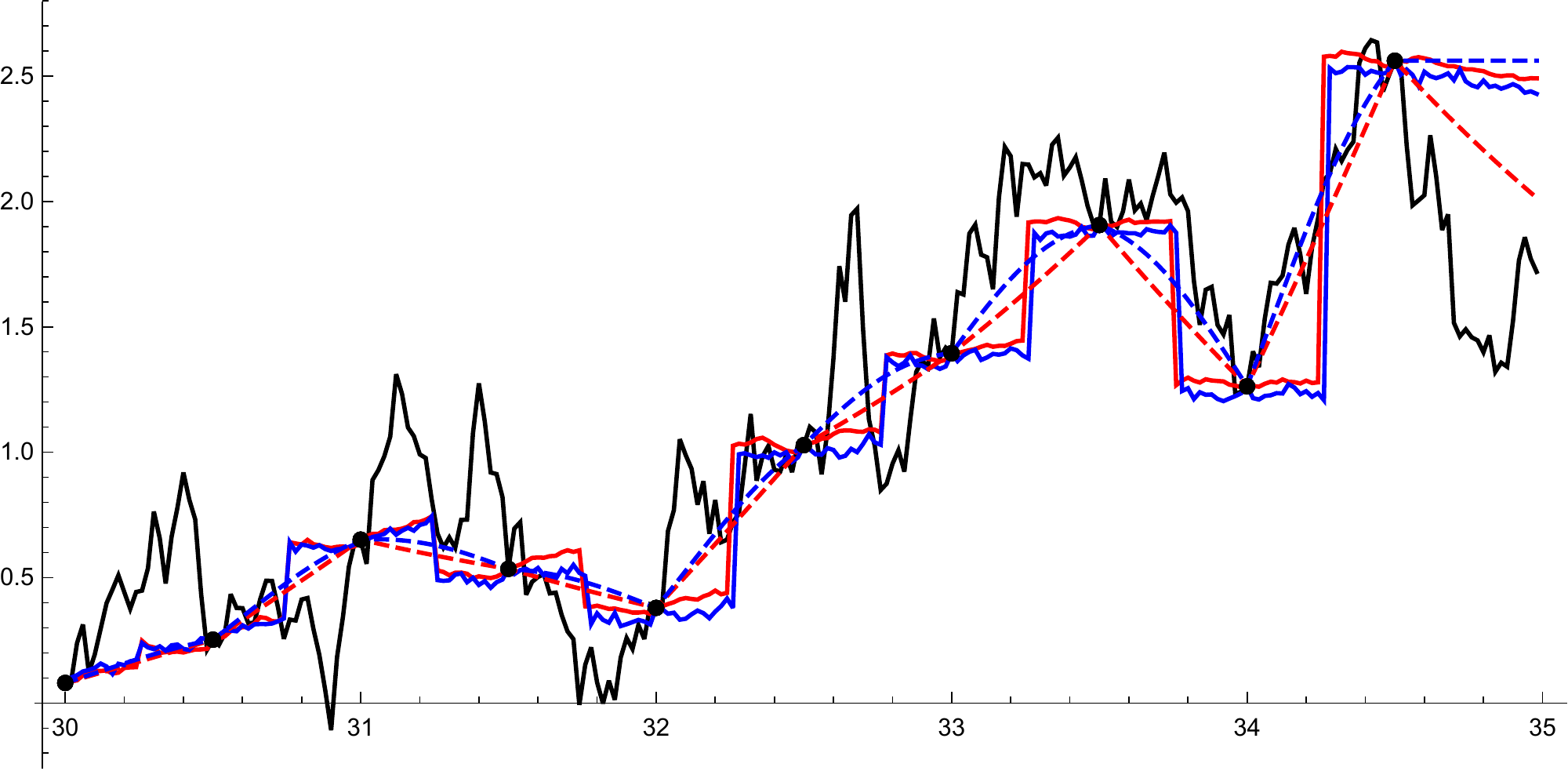}
    \caption{Interpolation of a Gaussian random process $X$. True trajectory $X(t)$ (black), predicted trajectories $\hat{X}_u$ (red, solid), $\hat{X}_c$ (blue, solid), $\hat{X}_{sk}$ (red, dashed), $\hat{X}_{e}$ (blue, dashed).}
    \label{T:Gauss:Fig:int}
\end{figure}
\begin{figure}[p]
    \centering
    \includegraphics[width=0.8\linewidth]{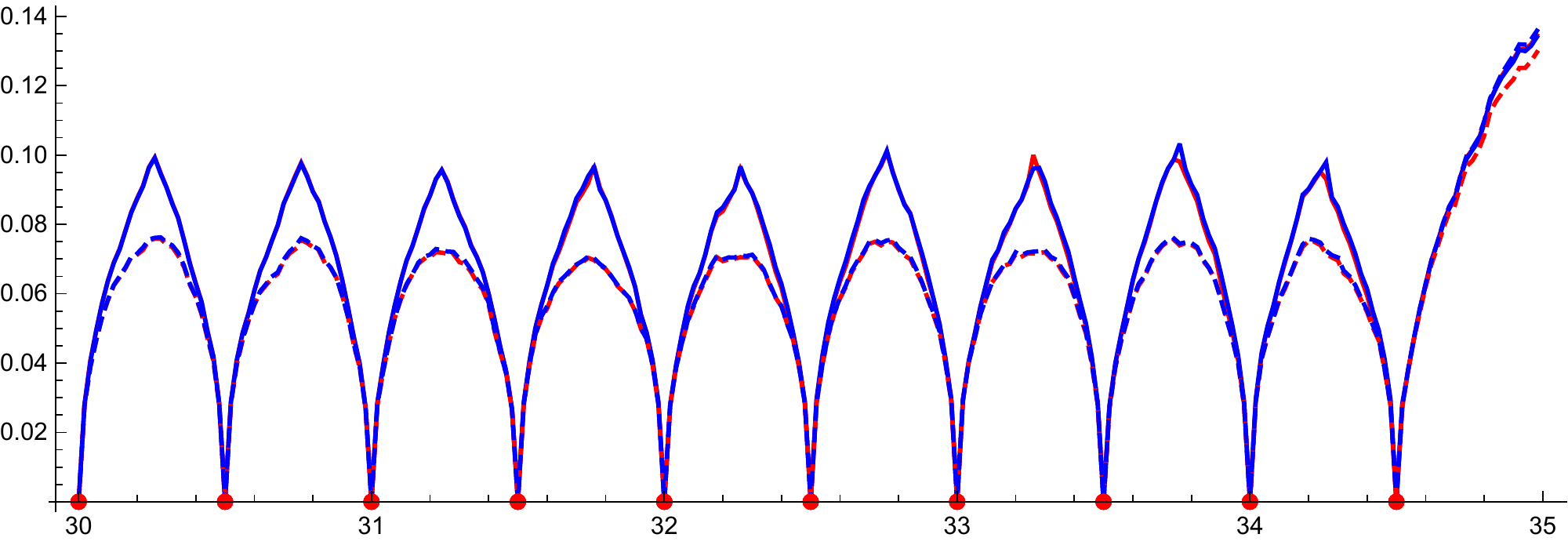}
    \caption{Interpolation of a Gaussian random process $X$. Excursion metric for predictors $\hat{X}_u$ (red, solid), $\hat{X}_c$ (blue, solid), $\hat{X}_{sk}$ (red, dashed), $\hat{X}_{e}$ (blue, dashed).}
    \label{G:Gauss:Fig:int}
\end{figure}
\begin{figure}[p]
    \centering
    \includegraphics[width=0.8\linewidth]{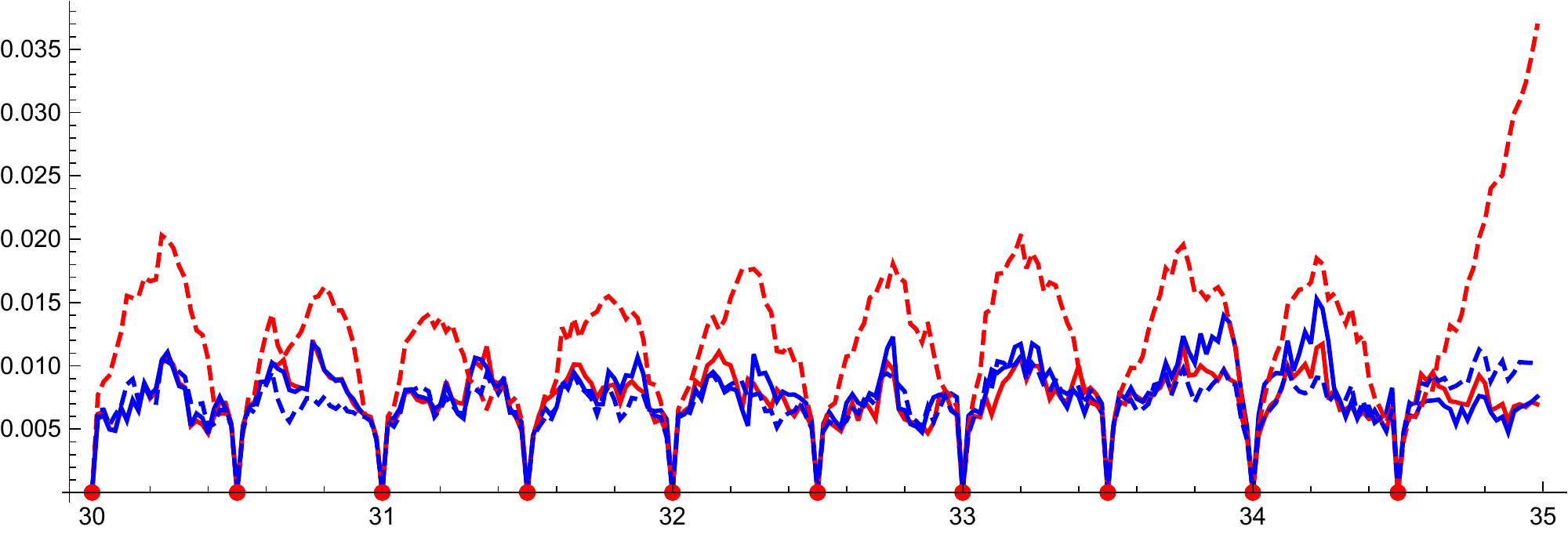}
    \caption{Interpolation of a Gaussian random process $X$. Wasserstein distance between $F(X)$ and predictors $F(\hat{X}_u)$ (red, solid), $F(\hat{X}_c)$ (blue, solid), $F(\hat{X}_{sk})$ (red, dashed), $F(\hat{X}_{e})$ (blue, dashed).}
    \label{W:Gauss:Fig:int}
\end{figure}

\begin{figure}[p]
    \centering
    \includegraphics[width=0.8\linewidth]{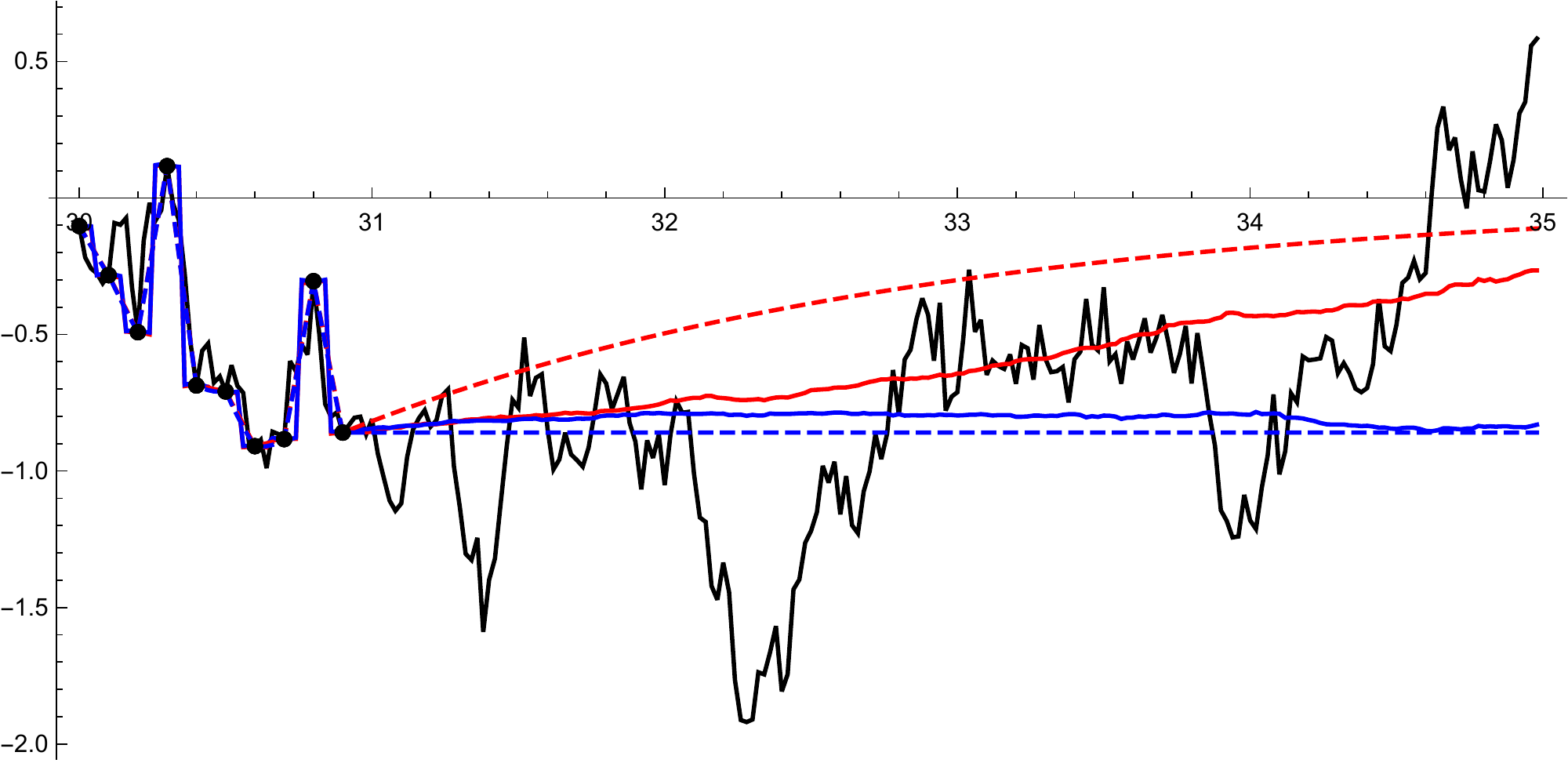}
     \caption{Extrapolation of a Gaussian random process $X$. True trajectory $X(t)$ (black), predicted trajectories $\hat{X}_u$ (red, solid), $\hat{X}_c$ (blue, solid), $\hat{X}_{sk}$ (red, dashed), $\hat{X}_{e}$ (blue, dashed).}
    \label{T:Gauss:Fig:extr}
\end{figure}
\begin{figure}[p]
    \centering
    \includegraphics[width=0.8\linewidth]{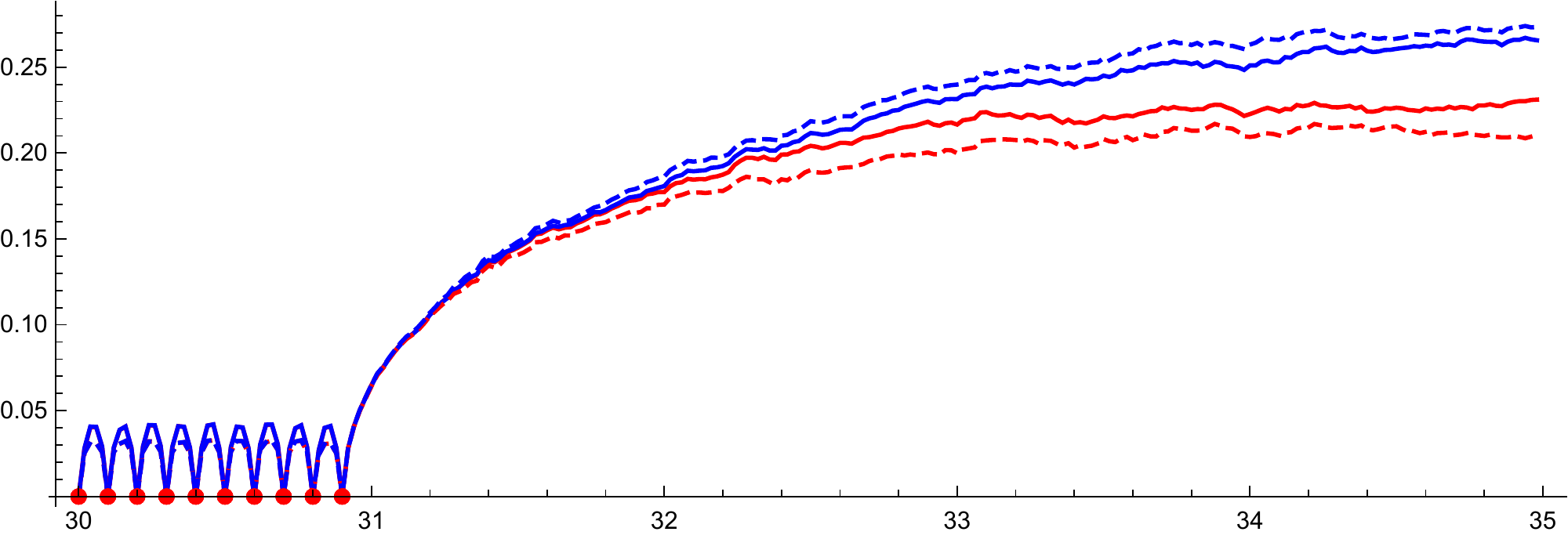}
    \caption{Extrapolation of a Gaussian random process $X$. Excursion metric for predictors $\hat{X}_u$ (red, solid), $\hat{X}_c$ (blue, solid), $\hat{X}_{sk}$ (red, dashed), $\hat{X}_{e}$ (blue, dashed).}
    \label{G:Gauss:Fig:extr}
\end{figure}
\begin{figure}[p]
    \centering
    \includegraphics[width=0.8\linewidth]{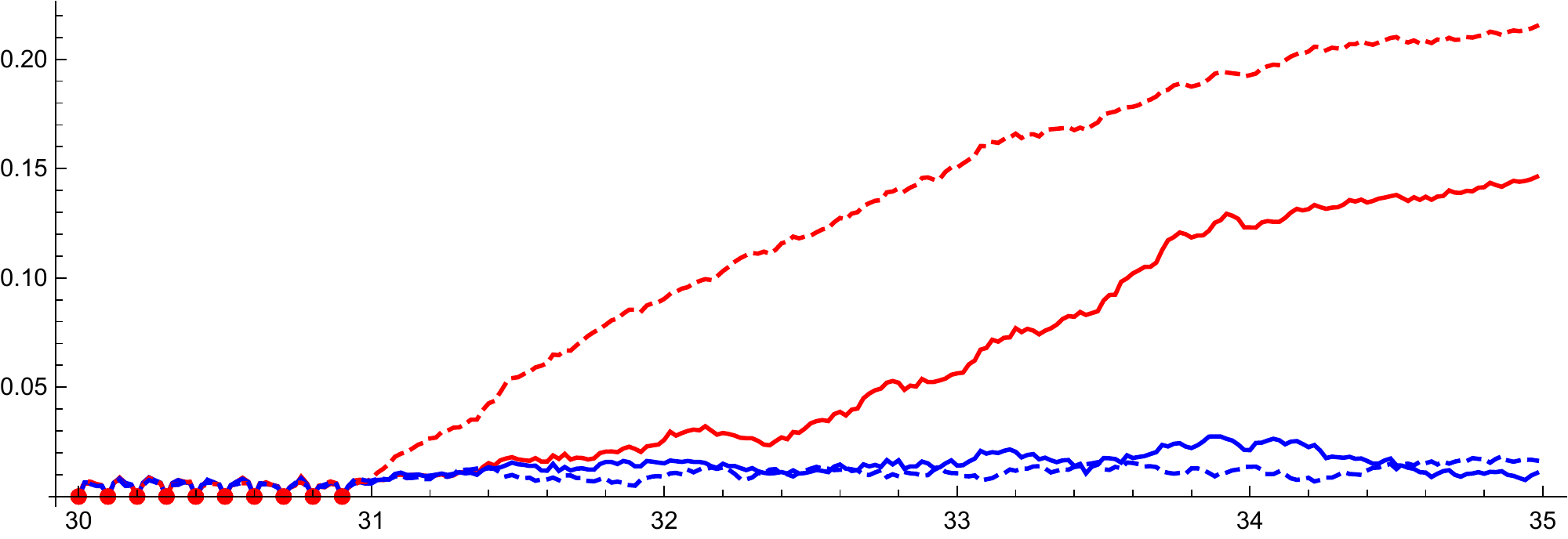}
    \caption{Extrapolation of a Gaussian random process $X$. Wasserstein distance between $F(X)$ and predictors $F(\hat{X}_u)$ (red, solid), $F(\hat{X}_c)$ (blue, solid), $F(\hat{X}_{sk})$ (red, dashed), $F(\hat{X}_{e})$ (blue, dashed).}
    \label{W:Gauss:Fig:extr}
\end{figure}

%--------------------------------------------------------------------------------------------------------------------------------------------------------------------
\subsection{$\alpha$-stable moving averages}

We consider the moving average process $X_\alpha=\{X_\alpha(t),\,t\in \Z_{h_1}\}$ given by
$X_\alpha(t)=\sum_{x\in \Z}m(t/{h_1}-x)\xi_\alpha(x)$ where $\xi_\alpha(x)$ are independent $S_\alpha(1,\beta,0)-$random variables and $m:\Z\to\R_+$ is a kernel function such that $\|m\|_\alpha:=\left(\sum_{x\in \Z}m^\alpha(x)\right)^{1/\alpha}<\infty.$   $X_\alpha$ is stationary with  marginal distribution $S_\alpha(\|m\|_\alpha,\beta,0)$.

Two cases of heaviness of the tails  are chosen:  Cauchy distribution ($\alpha=1,$ $\beta=0$) and L\'{e}vy distribution ($\alpha=0.5,$ $\beta=1$). 
 The kernel function is given by
$$m(x)=\begin{cases} e^{-0.02x}\left(1-e^{-0.02}\right)\left(1-e^{-5.02}\right)^{-1} \mathbbm{1}(x\in \Z\cap[0,251]),\quad&\alpha=1,\\
e^{-0.02x} \left(1-e^{-0.01}\right)^2\left(1-e^{-2.51}\right)^{-2}  \mathbbm{1}(x\in \Z\cap[0,251]),\quad&\alpha=0.5.
\end{cases}$$

It holds $\|m\|_\alpha=1$ and thus $X_\alpha(t)\sim S_\alpha(1,\beta,0).$ We choose $S_{0.5}(1,1,0)$ and $S_1(1,0,0)$ because there are simple analytical formulas for their c.d.f.'s :
\begin{equation}
    F_{X_{0.5}}(y)=\frac{1}{\sqrt{2\pi}}y^{-3/2}e^{-\frac{1}{2y}}\mathbbm{1}\{y>0\},\quad F_{X_1}(y)=\frac{1}{2}+\frac{1}{\pi}\arctan(y),\quad y\in \R.
\end{equation}
Moreover, $\E |X_\alpha(t)|=+\infty$ and $\E X_\alpha^2(t)=+\infty$ in both cases.

For each $T_f,$ $X_{0.5},$ and $X_{1},$ we find numerical solutions $\hat{\lambda}_u(t),$ and $\hat{\lambda}_c(t),$  $t\in \mathbb{Z}_{h_1}\cap[30,35] \setminus T_f$ of minimization problems \eqref{eq:minFctl_noConstraint}--\eqref{eq:minFctl}, which leads to predicted trajectories $\hat{X}_{0.5,u}, \hat{X}_{1,u}$ and $\hat{X}_{0.5,c}, \hat{X}_{1,c},$ respectively, see Figures \ref{T:Cauchy:Fig:int},\ref{T:Cauchy:Fig:extr},\ref{T:Levy:Fig:int}, and \ref{T:Levy:Fig:extr}.

One can observe that the predicted trajectory $\hat{X}_{1,c}$ in  Figure \ref{T:Cauchy:Fig:int} look like a step-wise functions and the extrapolated trajectories in Figure  \ref{T:Cauchy:Fig:extr} are quite volatile. In order to understand the source of these phenomena,  we apply two different methods of numerical solution. For $\alpha=1,$ we use stochastic subgadient descent method \eqref{eq:OnlineGr} and, for $\alpha=0.5,$ we apply Remark \ref{Rem:6} and existing optimization routine in R, avoiding the modification of subgradients.  

The predicted trajectories in Figures \ref{T:Levy:Fig:int} and \ref{T:Levy:Fig:extr} seem to fit the real trajectory better. Therefore, the existing minimization procedures can be more stable and accurate than the stochastic subgradient descent algorithm, which is nonetheless compensated  by much larger runtimes.

We repeat the simulation and prediction procedure 1000 times and  compute  the corresponding values of excursion metrics $E_{F_{X_\alpha}}(X_\alpha(t),\hat{X}_{\alpha,\lambda}(t))$, see Figures \ref{G:Cauchy:Fig:int}, \ref{G:Cauchy:Fig:extr}, \ref{G:Levy:Fig:int} and \ref{G:Levy:Fig:extr}. The corresponding  Wasserstein distances $\rho(F(X_\alpha(t),F(\hat{X}_{\alpha,\lambda}(t)))$  are given in Figures \ref{W:Cauchy:Fig:int}, \ref{W:Cauchy:Fig:extr}, \ref{W:Levy:Fig:int} and \ref{W:Levy:Fig:extr}.

In Figures \ref{G:Gauss:Fig:extr}, \ref{G:Cauchy:Fig:extr}, and \ref{G:Levy:Fig:extr}, the excursion metrics tend asymptotically to value $1/3$ as the distance to the last observed point increases. This alludes to Example \ref{ex:GiniIndependent} and shows that predictor $\hat{X}_{\alpha,\lambda}(t)$ and true random variable $X_{\alpha}(t)$ become asymptotically independent (for large $t$).

{\begin{figure}[p]
    \centering
    \includegraphics[width=0.8\linewidth]{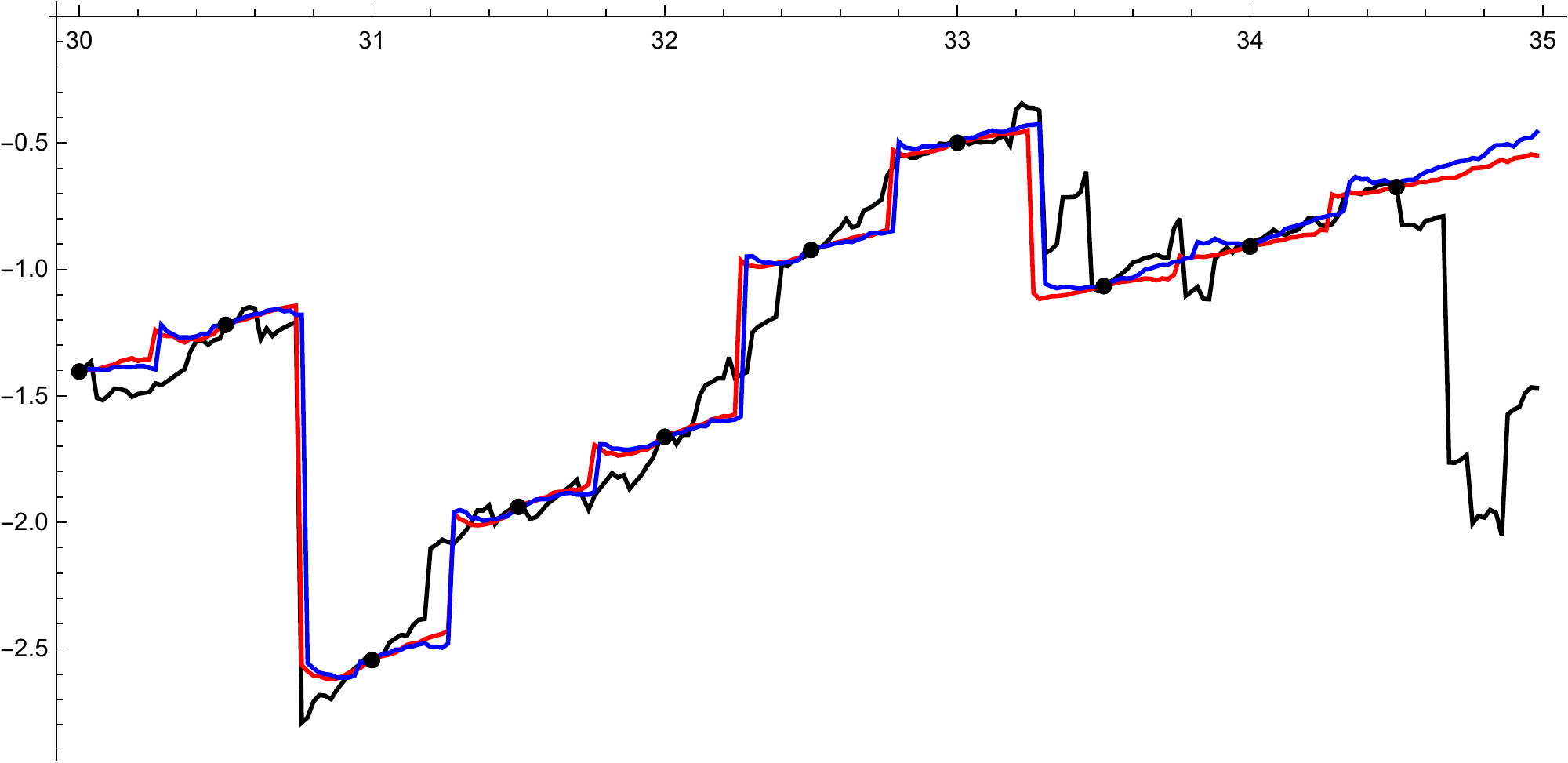}
    \caption{Interpolation of a moving average $X_{1}$ with Cauchy distributed marginals: True trajectory $X_{1}(t)$ (black), predicted trajectories $\hat{X}_{1,u}$ (red) and $\hat{X}_{1,c}$ (blue).}
    \label{T:Cauchy:Fig:int}
\end{figure}
\begin{figure}[p]
    \centering
    \includegraphics[width=0.8\linewidth]{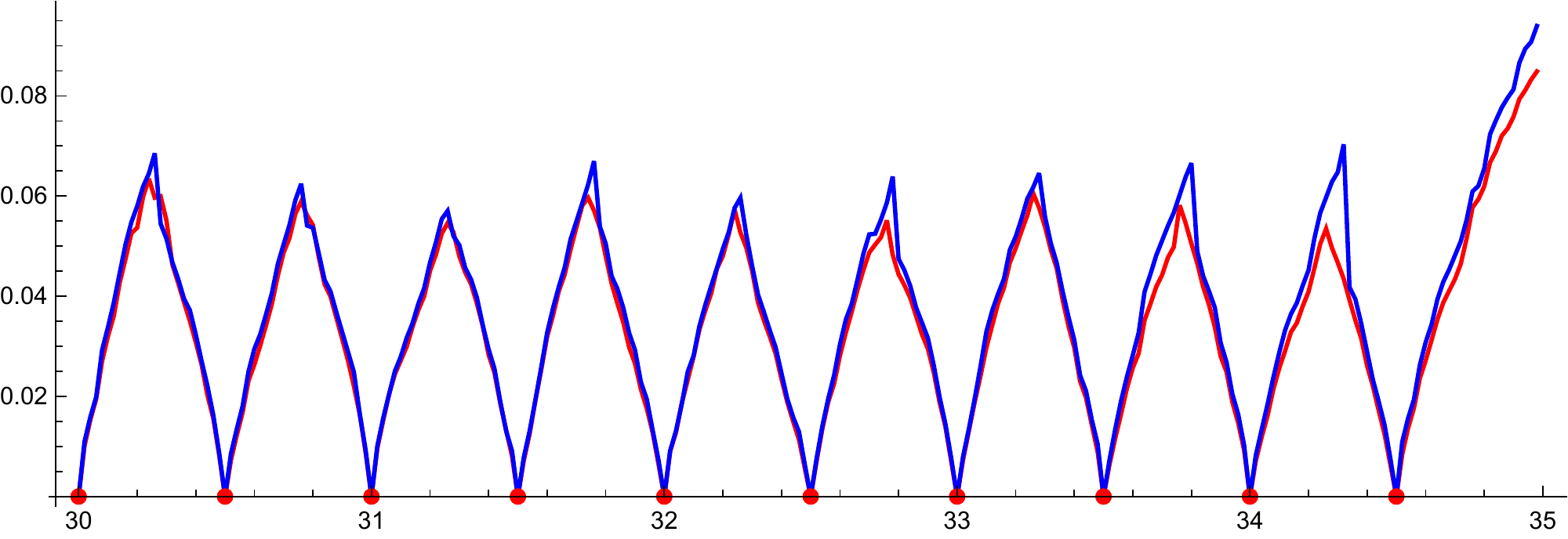}
    \caption{Interpolation of a moving average $X_{1}$ with Cauchy distributed marginals:  Excursion metric for predictors $\hat{X}_{1,u}$ (red) and $\hat{X}_{1,c}$ (blue).}
    \label{G:Cauchy:Fig:int}
\end{figure}
\begin{figure}[p]
    \centering
    \includegraphics[width=0.8\linewidth]{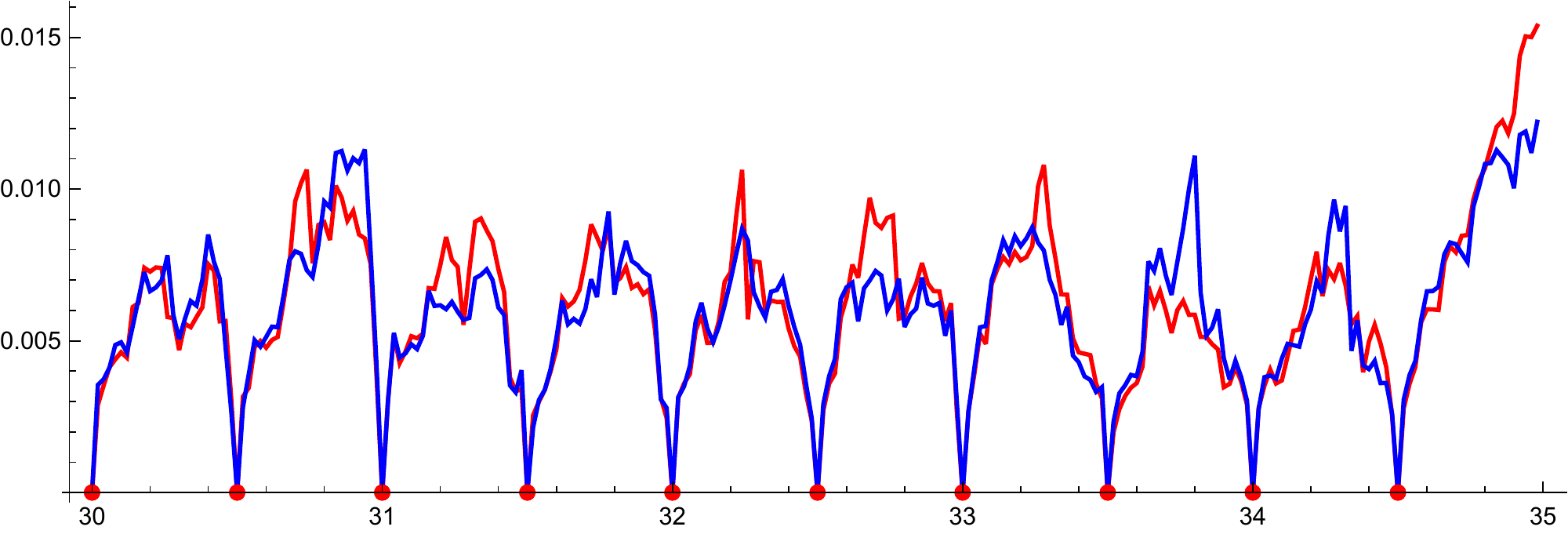}
    \caption{Interpolation of a moving average $X_{1}$ with Cauchy distributed marginals: Wasserstein distance between $F(X_{1})$ and predictors $F(\hat{X}_{1,u})$ (red), $F(\hat{X}_{1,c})$ (blue).}
    \label{W:Cauchy:Fig:int}
\end{figure}
}

{
\begin{figure}[p]
    \centering
    \includegraphics[width=0.8\linewidth]{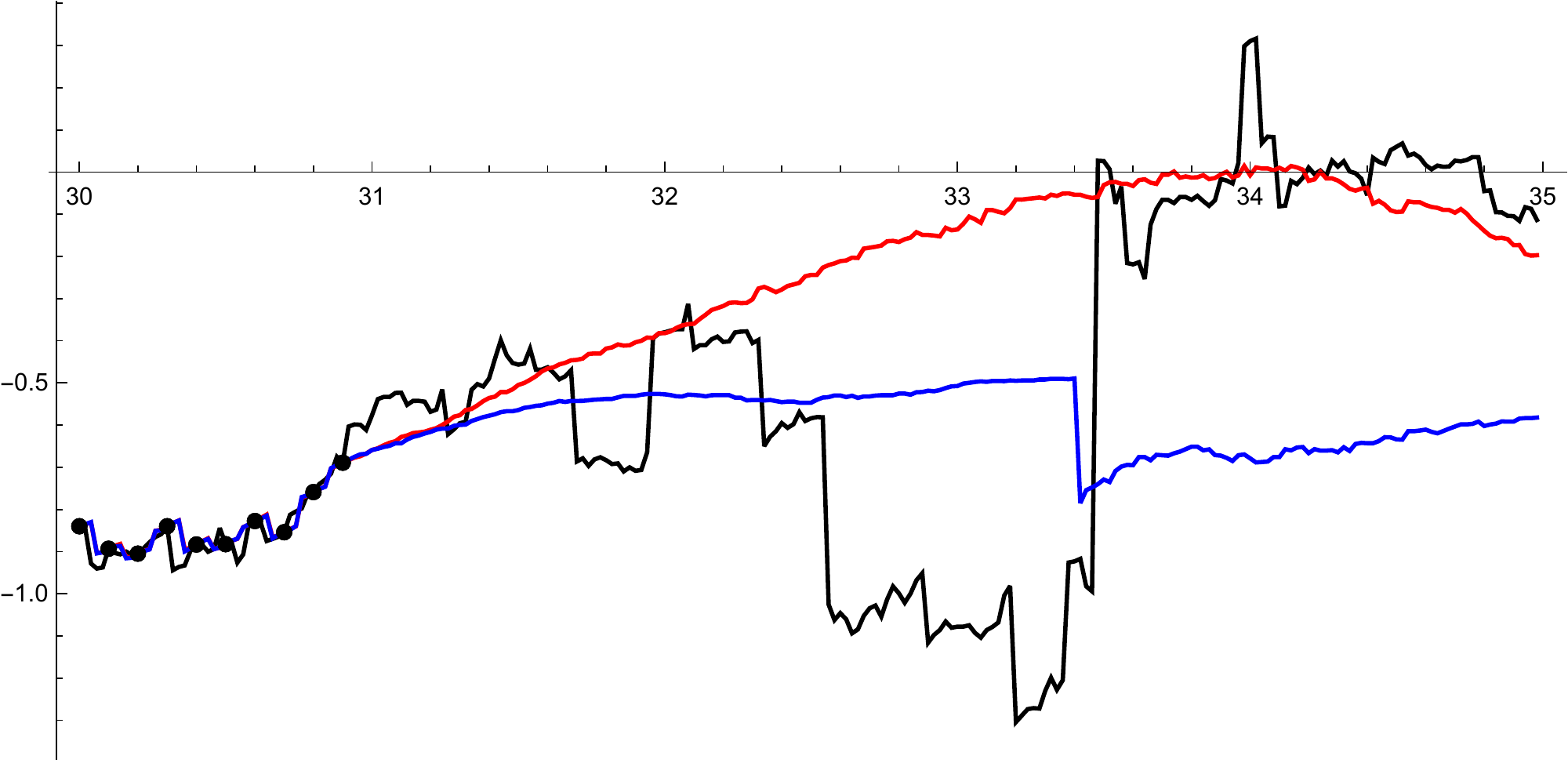}
    \caption{Extrapolation of a moving average $X_{1}$ with Cauchy distributed marginals: True trajectory $X_{1}(t)$ (black), predicted trajectories $\hat{X}_{1,u}$ (red) and $\hat{X}_{1,c}$ (blue).}
    \label{T:Cauchy:Fig:extr}
\end{figure}
\begin{figure}[p]
    \centering
    \includegraphics[width=0.8\linewidth]{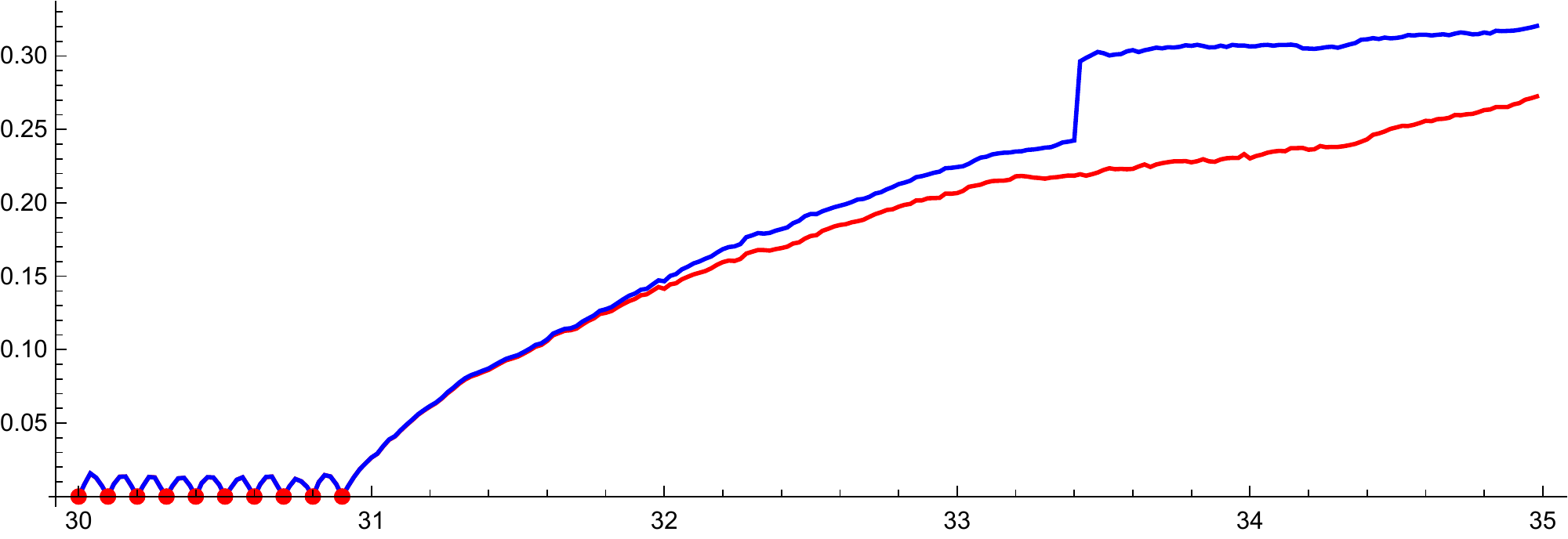}
    \caption{Extrapolation of a moving average $X_{1}$ with Cauchy distributed marginals: Excursion metric for predictors $\hat{X}_{1,u}$ (red) and $\hat{X}_{1,c}$ (blue).}
    \label{G:Cauchy:Fig:extr}
\end{figure}
\begin{figure}[p]
    \centering
    \includegraphics[width=0.8\linewidth]{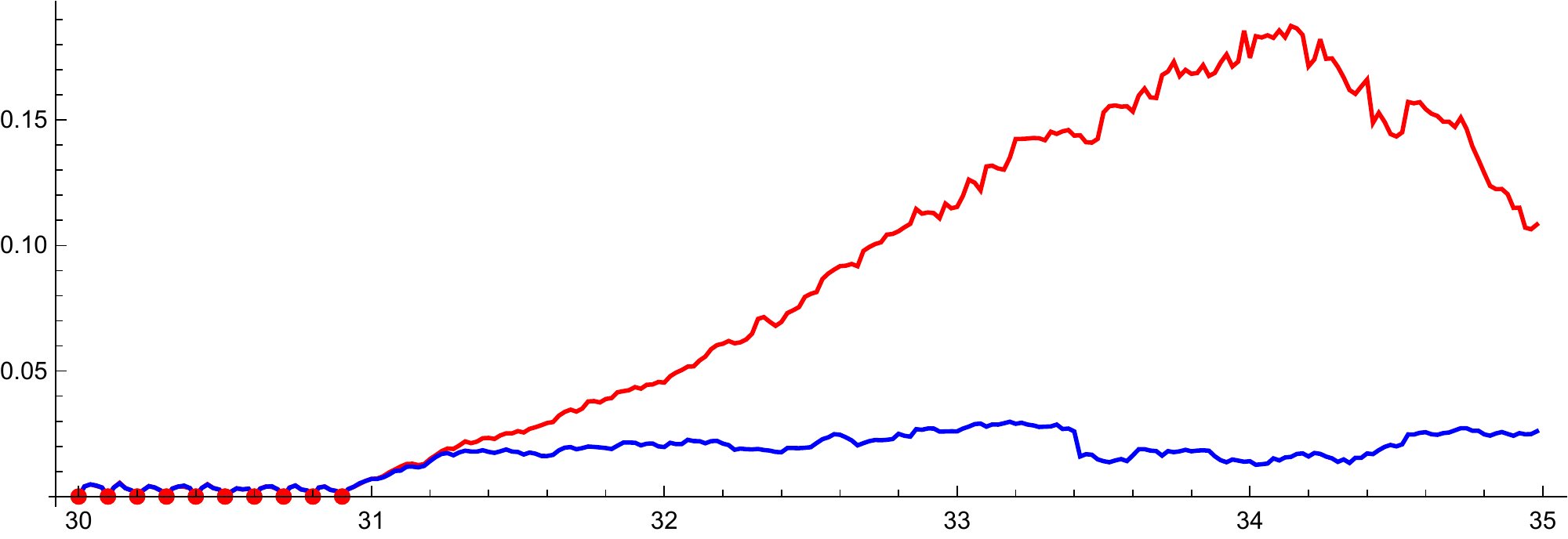}
    \caption{Extrapolation of a moving average $X_{1}$ with Cauchy distributed marginals: Wasserstein distance between $F(X_{1})$ and predictors $F(\hat{X}_{1,u})$ (red), $F(\hat{X}_{1,c})$ (blue).}
    \label{W:Cauchy:Fig:extr}
\end{figure}
}

{\begin{figure}[p]
    \centering
    \includegraphics[width=0.8\linewidth]{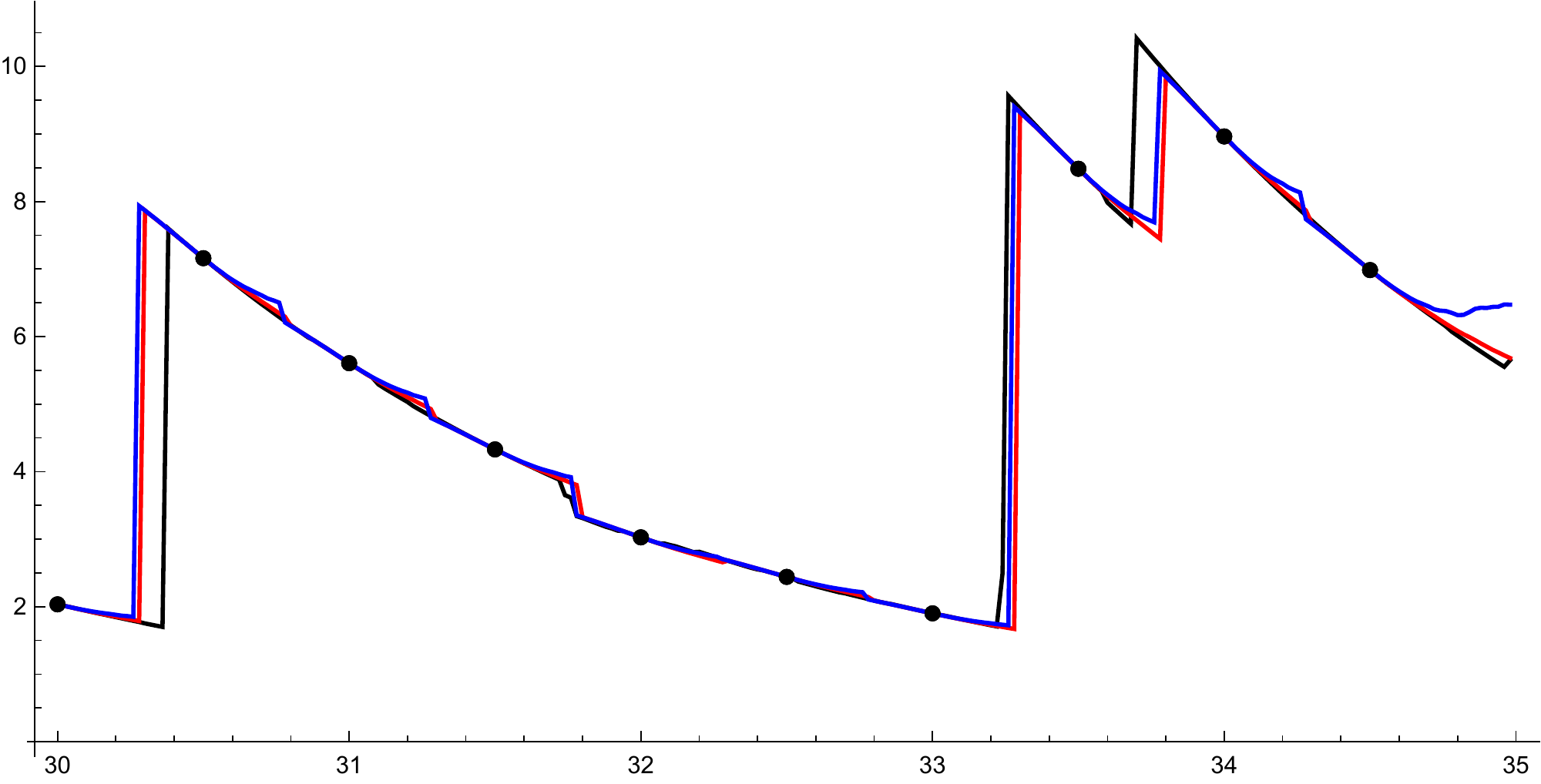}
    \caption{Interpolation of a moving average $X_{0.5}$ with L\'{e}vy distributed marginals: True trajectory $X_{0.5}(t)$ (black), predicted trajectories $\hat{X}_{0.5,u}$ (red) and $\hat{X}_{0.5,c}$ (blue).}
    \label{T:Levy:Fig:int}
\end{figure}
\begin{figure}[p]
    \centering
    \includegraphics[width=0.8\linewidth]{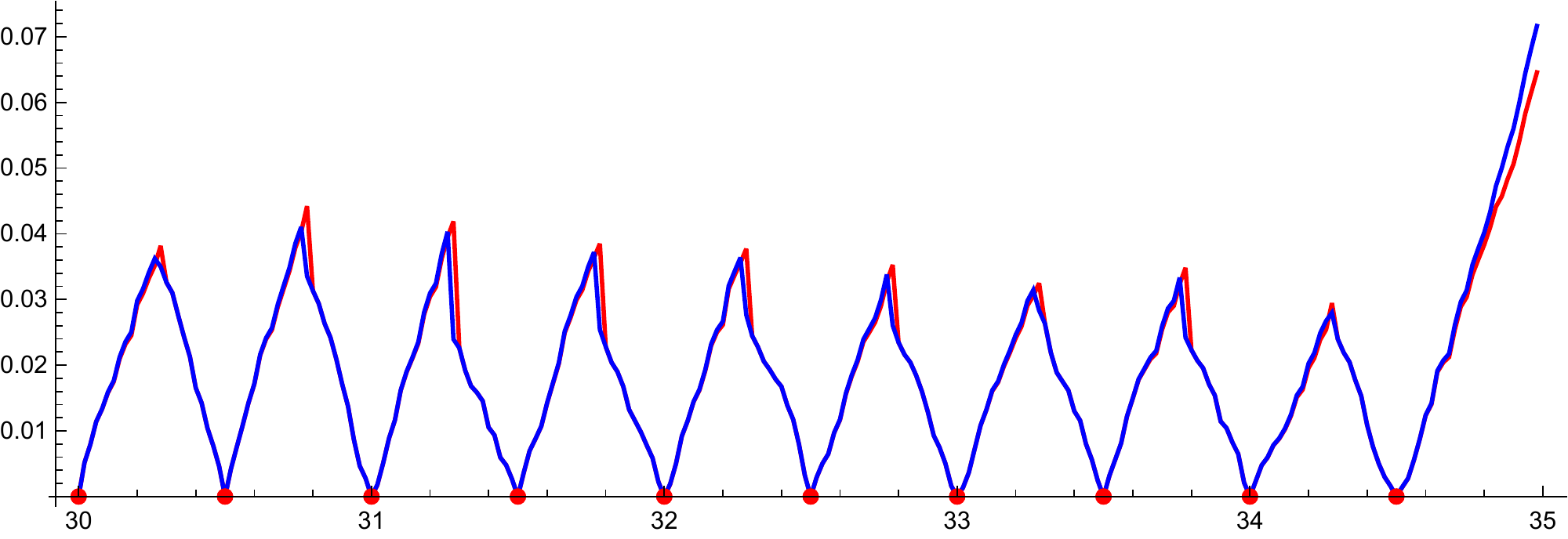}
    \caption{Interpolation of a moving average $X_{0.5}$ with L\'{e}vy distributed marginals: Excursion metric for predictors $\hat{X}_{0.5,u}$ (red) and $\hat{X}_{0.5,c}$ (blue).}
    \label{G:Levy:Fig:int}
\end{figure}
\begin{figure}[p]
    \centering
    \includegraphics[width=0.8\linewidth]{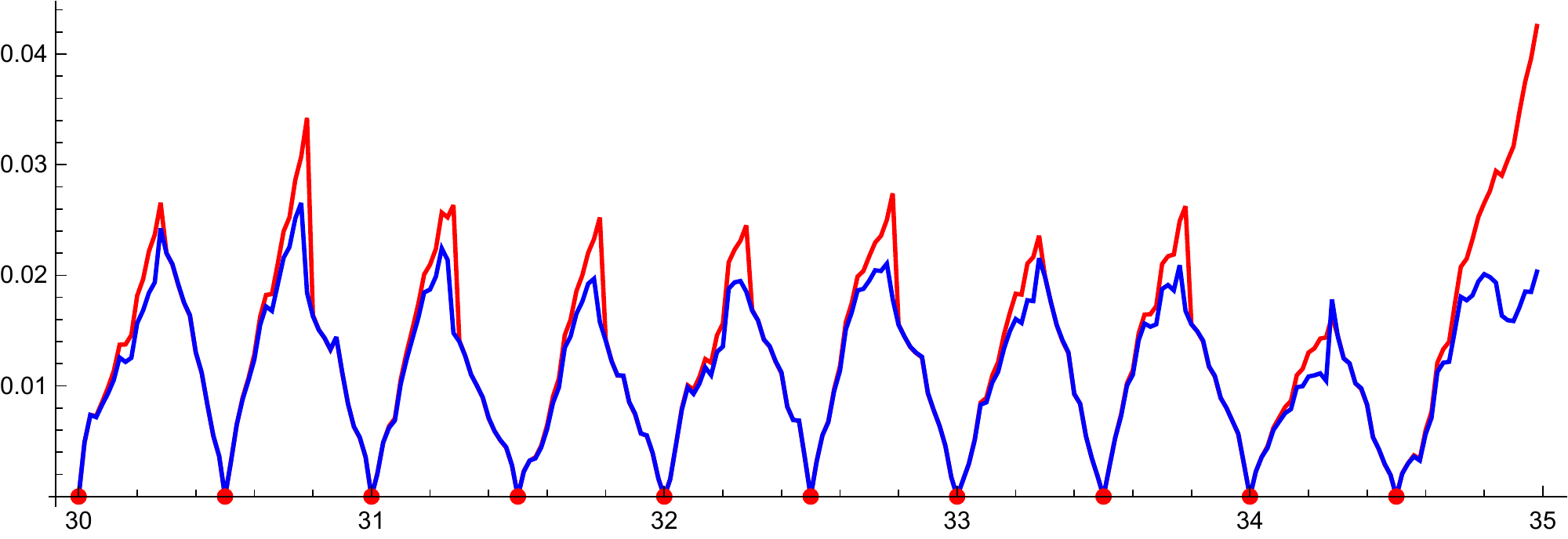}
    \caption{Interpolation of a moving average $X_{0.5}$ with L\'{e}vy distributed marginals: Wasserstein distance between $F(X_{0.5})$ and predictors $F(\hat{X}_{0.5,u})$ (red), $F(\hat{X}_{0.5,c})$ (blue).}
    \label{W:Levy:Fig:int}
\end{figure}
}

{
\begin{figure}[p]
    \centering
    \includegraphics[width=0.8\linewidth]{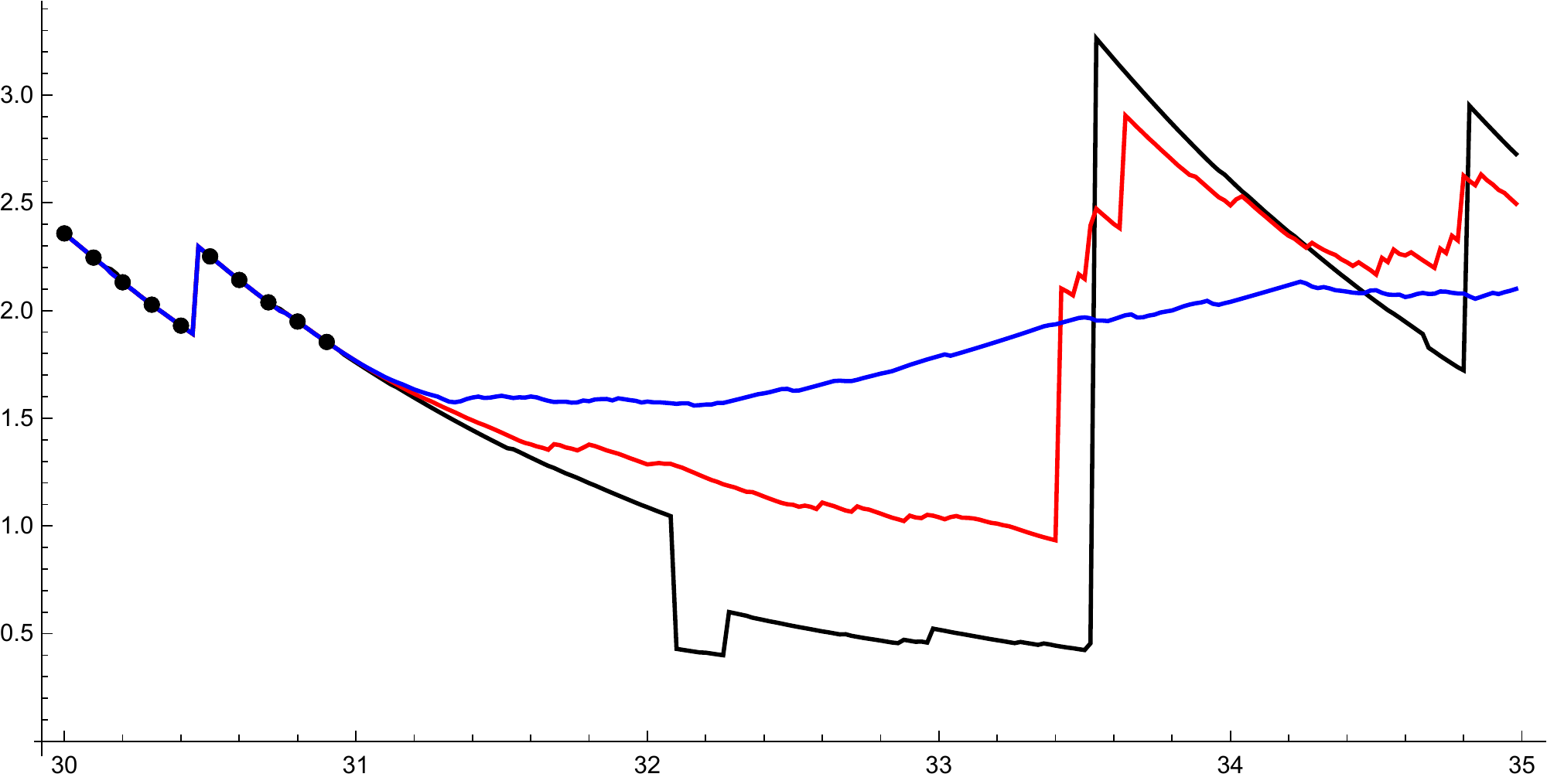}
    \caption{Extrapolation of a moving average $X_{0.5}$ with L\'{e}vy distributed marginals: True trajectory $X_{0.5}(t)$ (black), predicted trajectories $\hat{X}_{0.5,u}$ (red) and $\hat{X}_{0.5,c}$ (blue).}
    \label{T:Levy:Fig:extr}
\end{figure}
\begin{figure}[p]
    \centering
    \includegraphics[width=0.8\linewidth]{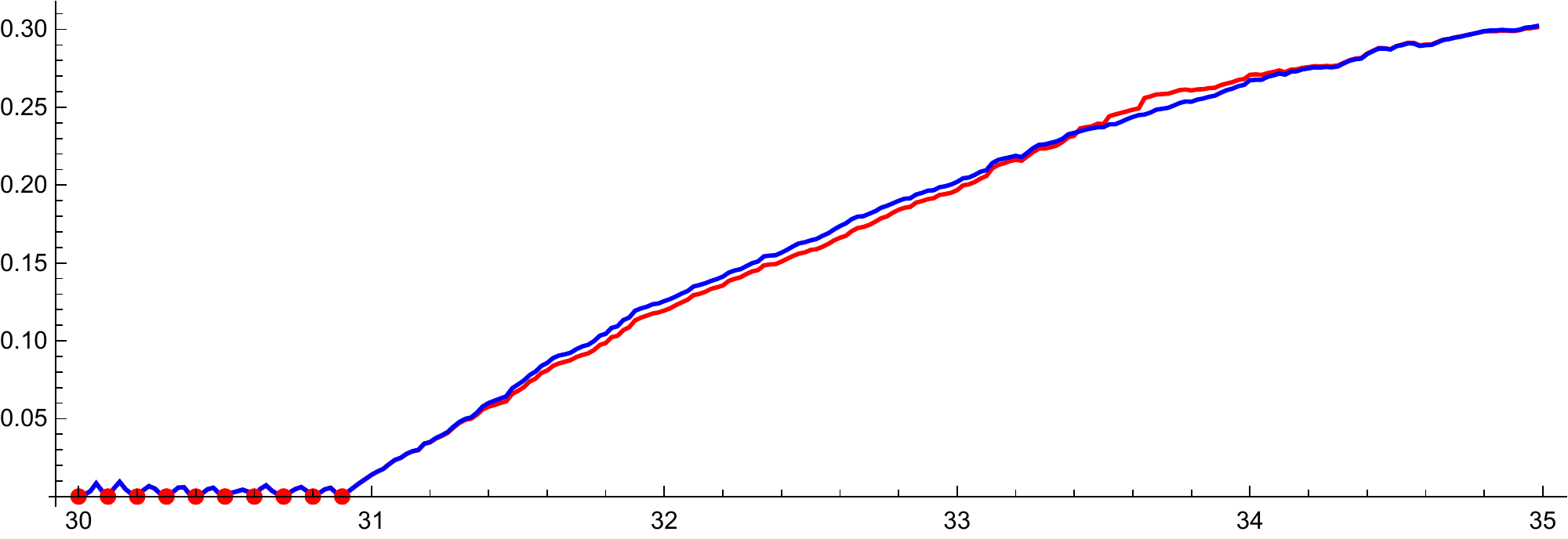}
    \caption{Extrapolation of a moving average $X_{0.5}$ with L\'{e}vy distributed marginals: Excursion metric for predictors $\hat{X}_{0.5,u}$ (red) and $\hat{X}_{0.5,c}$ (blue).}
    \label{G:Levy:Fig:extr}
\end{figure}
\begin{figure}[p]
    \centering
    \includegraphics[width=0.8\linewidth]{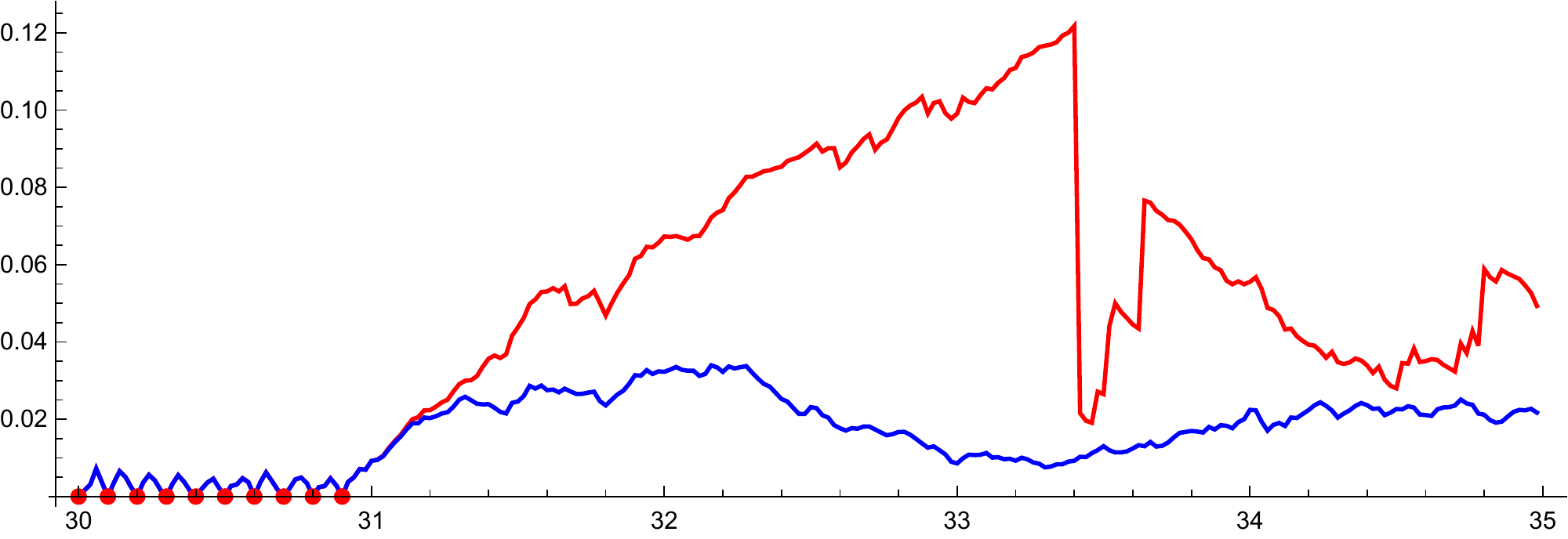}
    \caption{Extrapolation of a moving average $X_{0.5}$ with L\'{e}vy distributed marginals: Wasserstein distance between $F(X_{0.5})$ and predictors $F(\hat{X}_{0.5,u})$ (red), $F(\hat{X}_{0.5,c})$ (blue).}
    \label{W:Levy:Fig:extr}
\end{figure}
}
%-------------------------------------------------------------------------------------------------------------------------------------------------------------------------------------------------------------
\subsection{Autoregressive stationary process}
A random process which can be  nicely predicted via a linear forecast is an autoregressive time series $AR(p)$ given by 
$$X(t)=\vf_{1}X(t-ph)+\vf_{2}X(t-ph+h)+\ldots+\vf_{p}X(t-h)+\xi_t,$$
where $\xi_{t},t\in h\Z$ are independent random variables and $\vf_k,k=1,\ldots,p$ are regression coefficients.

We examine our method on a simulated trajectory of $X(t),t\in \mathbb{T}_0\cup T_f$ with $T_f=\{30.0,30.1,30.2\}$, $h=0.02$ and $n=p=3.$ We take $\vf_1=0.1,$ $\vf_2=0.25,$ $\vf_3=0.5$ in order to get a stationary $AR(3)$ process, and set $\xi_t$ being  standard Student t-distributed  with  $\nu=0.8$ degrees of freedom which is infinitely divisible with $\E |\xi_t|=\infty.$ 
Indeed, it is not hard to verify that the roots of the equation $\sum _{j=1}^{p}\varphi_{j}z^{j}=1$, $z\in\mathbb{C}$ lie outside the unit circle, and the tail probabilities of $\xi_t$ are regularly varying, cf. conditions in \cite{Cline}.

\begin{remark}
In this case, we do not know the exact marginal distribution of $X$. However, we can use the excursion predictor $\hat{X}_\lambda$ from Definition \ref{def:excurs:pred} with c.d.f. function $F$ having the same support as $X(t).$ If $\lambda$ is defined via \eqref{unconst:eq1} or \eqref{eq:minFctl_noConstraint}, one can show that $\hat{X}_\lambda$ is a consistent estimator of $X(t)$ as well.
In the case of constrained minimization  \eqref{eq:minFctl}--\eqref{eq:minFctl1}, the distribution of $F(X(t))$ is not uniform any more, and the relation $\E F^2(\hat{X}_\lambda)-\E[F(\hat{X}_\lambda) \vee Y]$ does not correspond to the squared 2-Wasserstein distance $\rho^2(Law[F(\hat{X}_\lambda)],Law[F(X(t))]).$ By triangle inequality, it holds 
$$\rho(Law(\hat{X}_\lambda),Law[F(X(t))])\leq \rho(Law[F(\hat{X}_\lambda)],U[0,1])+\rho(U[0,1],Law[F(X(t))]),$$
and hence the minimization of $\rho(Law[F(\hat{X}_\lambda)],U[0,1])$ leads to the approximative minimization of  \\ $\rho(Law(\hat{X}_\lambda),Law[F(X(t))]),$ if $\rho(U[0,1],Law[F(X(t))])$ is small enough.
So, we expect our predictor estimator to be robust regarding the choice of function $F.$
\end{remark}

 Therefore, we use the  excursion metric $E_F$ with c.d.f. $F=F_{\hat{\theta}}$ being close to the true marginal distribution of $X$ and having the same support. We choose $F_\theta,$ $\theta=(\mu,\sigma,\nu)$ from the parametric family of Student t-distributions $ST(\mu,\sigma,\nu).$ Based on a simulated trajectory of $X,$ we find that  $F_{\hat{\theta}}$ with $\hat{\theta}=(0,10,0.7)$ best approximates the marginals of $X$. 

In this example, we do not extrapolate the trajectories of $X$ on a wide interval, but we study the performance of the minimization algorithm. That is why the prediction coefficients $\lambda(t)=(\lambda_1(t),\lambda_2(t),\lambda_3(t))$ are computed at points $t\in\{30.3,30.4,30.5,30.6\}$. We do this for minimization problems \eqref{eq:minFctl_noConstraint} and \eqref{eq:minFctl} via stochastic subgradient descent and via standard minimization methods implemented in R language resulting in $\hat{\lambda}_k^g$ and $\hat{\lambda}_k^r,$ $k=u,c$, respectively. The results are given in Tables \ref{TB1} (unconstrained) and \ref{TB2} (constrained).

One can observe that the minimal values of $\bar{\Phi}_2$ and $\bar{\Phi}_3$ increase slightly when the prediction point $t$ moves away from the forecast sample $T_f.$ We see also that the standard implemented minimization methods reach the lower minimal values of $\bar{\Phi}_2$ and $\bar{\Phi}_3.$ However, the stochastic subgradient descent method is much faster and the differences  $\|\hat{\lambda}_u^g(t)-\hat{\lambda}_u^r(t)\|_2$ and $\|\hat{\lambda}_c^g(t)-\hat{\lambda}_c^r(t)\|_2$ are small for the prediction points which are close to the forecast sample.

The values of $\hat{\lambda}_u^g(t)=(0.10490,\, 0.24573,\, 0.49832)$ and $\hat{\lambda}_c^g(t)=(0.12154,\, 0.23066,\, 0.48210)$ at point $t=30.3$ are very close to regression coefficients $\vf.$ Therefore, one can use the reliable prediction method via excursion metric $E_{F_X}$  in the case when   the marginal distribution is not known a-priori and has to be statistically assessed.

\begin{table}[ht]
\footnotesize
\centering
\begin{tabular}{cccccc}
$t$& $\hat{\lambda}_u^g(t)$ & $\bar{\Phi}_2(\hat{\lambda}_u^g(t))$ & $\hat{\lambda}_u^r(t)$& $\bar{\Phi}_2(\hat{\lambda}_u^r(t))$& $\|\hat{\lambda}_u^g(t)-\hat{\lambda}_u^r(t)\|_2$\\
\hline
  30.3 & (0.10490,\, 0.24573,\, 0.49832) & 0.04423 & (0.10832,\, 0.24299,\, 0.49852)& 0.04423 & 0.00438 \\ 
  30.4 & (0.05293,\, 0.22722,\, 0.49585) & 0.06358 & (0.05914,\, 0.22262,\, 0.49288) & 0.06358 & 0.00827 \\ 
  30.5 & (0.04532,\, 0.19385,\, 0.46395) & 0.08082 & (0.06593,\, 0.18834,\, 0.44610) & 0.08076 & 0.02782 \\ 
  30.6 & (0.15467,\, 0.19230,\, 0.28920) & 0.09702 & (0.05452,\, 0.16589,\, 0.41035) & 0.09579 & 0.15938 \\ 
  $\vf$ & (0.10000,\, 0.25000,\, 0.50000) & &   (0.10000,\, 0.25000,\, 0.50000) & & \\ 
   \hline
\end{tabular}
\caption{AR(3) stationary process. Extrapolation coefficients $\hat{\lambda}_u^g$ and $\hat{\lambda}_u^r$ for unconstrained minimization.}
\label{TB1}
\end{table}

\begin{table}[ht]
\centering
\footnotesize
\begin{tabular}{cccccc}
$t$& $\hat{\lambda}_c^g(t)$ & $\bar{\Phi}_3(\hat{\lambda}_c^g(t))$ & $\hat{\lambda}_c^r(t)$& $\bar{\Phi}_3(\hat{\lambda}_c^r(t))$& $\|\hat{\lambda}_c^g(t)-\hat{\lambda}_c^r(t)\|_2$\\
\hline
  30.3 & (0.12154,\, 0.23066,\, 0.48210) & 0.04784 & (0.12081,\, 0.24737,\, 0.47881) & 0.04770 & 0.01705 \\ 
  30.4 & (0.05638,\, 0.23474,\, 0.50164) & 0.05212 & (0.07559,\, 0.22932,\, 0.48422) & 0.05207 & 0.02650 \\ 
  30.5 & (0.33977,\, 0.26291,\, 0.39731) & 0.08588 & (0.10881,\, 0.18561,\, 0.43390) & 0.07623 & 0.24629 \\ 
  30.6 & (-0.04932,\, 0.21927,\, 0.53541) & 0.08633 & (0.08251,\, 0.17675,\, 0.42702) & 0.08513 & 0.17588 \\ 
   $\vf$ & (0.10000,\, 0.25000,\, 0.50000) & &   (0.10000,\, 0.25000,\, 0.50000) & & \\ 
   \hline
\end{tabular}
\caption{AR(3) stationary process. Extrapolation coefficients $\hat{\lambda}_c^g$ and $\hat{\lambda}_c^r$ for constrained minimization.}
\label{TB2}
\end{table}

\section{Summary}
We introduced the new predictors for random variables, processes and fields with possibly infinite moments via the minimization of a functional based on excursion sets. We explored
several advantages of our excursion predictors using theoretical results and computational studies.  Namely, they are computationally fast, consistent for stochastically continuous random fields, and work for random fields without finite moments. The research presented in this paper is introductory and covers only the first important properties of our methods. These results reveal  a  great potential for many real world applications. A further theoretical investigation of our methods including the uniqueness of solutions, accuracy of prediction as well as the improvement of computation routines will be the topic of our next papers.

\bibliographystyle{abbrv}
\bibliography{MakoginLit}
\end{document}